\documentclass[11pt,amsfonts]{article}
\usepackage{mathrsfs,amssymb,amsthm}
\usepackage{graphicx}
\usepackage{latexsym}
\usepackage{amssymb}
\usepackage{amsmath}
\usepackage{enumerate}
\usepackage{layout}
\usepackage{color}

\makeatletter
	
	\@addtoreset{equation}{section}
\makeatother

\newtheorem{prop}{Proposition}[section]
\newtheorem{lemma}[prop]{Lemma}

\newtheorem{corollary}[prop]{Corollary}
\newtheorem{theorem}[prop]{Theorem}
\newtheorem{remark}[prop]{Remark}
\newtheorem{ass}[prop]{Assumption}
\newtheorem{example}[prop]{Example}

\begin{document}
\renewcommand{\theenumi}{\rm (\roman{enumi})}
\renewcommand{\labelenumi}{\rm \theenumi}

\allowdisplaybreaks
\renewcommand{\thefootnote}{\fnsymbol{footnote}}
\title{Characterization of the convergence in total variation  and extension of the Fourth Moment Theorem to invariant measures of diffusions}
\author{ Seiichiro Kusuoka$^{1}$\footnote{Supported by JSPS KAKENHI Grant number 25800054.} 
and %
Ciprian A. Tudor$^2$
\vspace*{0.1in} \\
$^{1}$Graduate School of Natural Science and Technology, Okayama University \\
3-1-1 Tsushima-naka, Kita-ku, Okayama 700-8530 Japan\\
kusuoka@okayama-u.ac.jp \vspace*{0.1in} \\
 $^{2}$ Laboratoire Paul Painlev\'e, Universit\'e de Lille 1\\
 F-59655 Villeneuve d'Ascq, France.\\
 \quad tudor@math.univ-lille1.fr\vspace*{0.1in}}
\maketitle

\begin{abstract}
We give necessary and sufficient conditions to characterize the convergence in distribution of a sequence of arbitrary random variables  to a probability distribution which is the invariant measure of a diffusion process. This class of target distributions includes the most known continuous probability distributions. Precisely speaking, we characterize the convergence in total variation to target distributions which are not Gaussian or Gamma distributed, in terms of the Malliavin calculus and of the coefficients of the associated diffusion process.  We also prove that, among the distributions whose associated  squared diffusion coefficient is a polynomial of second degree (with some restrictions on its coefficients), the only possible limits of sequences of multiple integrals are the Gaussian and the Gamma laws. 
\end{abstract}

\vskip0.2cm

{\bf 2010 AMS Classification Numbers:}  60F05, 60H05, 91G70.

 \vskip0.2cm


\section{Introduction}

In the seminal paper \cite{NuPe}, Nualart and Peccati discovered a surprising central limit theorem (called {\it The Fourth Moment Theorem}) for sequences of multiple stochastic integrals in a Wiener chaos of a fixed order. This result says that the  convergence   in distribution of such a sequence of random variables to the standard normal law is actually equivalent to the  convergence of only the sequence of  their fourth moments. A multidimensional version of this result has been given in \cite{PeTu}. Since the publication of these two pathbreaking papers, many improvements and developments on this theme have been considered. Among them is the work \cite{NuOr}, giving a new proof only based on Malliavin calculus and the use of integration by parts on the Wiener space.

Another pathbreaking paper is the work \cite{NoPe1} by Nourdin and Peccati in which the authors bring together Stein's method with Malliavin calculus and obtain useful estimates for the distance between the law of an arbitrary random variable and the Gaussian distribution in terms of the Malliavin calculus.  It turns out that Stein's method and Malliavin calculus fit together admirably well, and that their interaction has led to some remarkable new results involving central and non-central limit theorems for functionals of infinite-dimensional Gaussian fields. We refer to the recent monographs  \cite{NPbook} for an overview of the existing literature and to \cite{T1} for various applications of the Stein's method and Malliavin calculus to limit theorems and statistics. 

There is also a version of the Fourth Moment Theorem having the Gamma distribution as the target distribution (see \cite{NoPe4}). The convergence of a sequence of multiple stochastic integrals toward a Gamma distribution is characterized by the convergence of the sequences of the  third and fourth moments; alternatively, one can also characterize the convergence to the Gamma law in terms of the Malliavin derivatives.

In the paper \cite{KuTu} we obtained bounds between the distance of an arbitrary random variable and target distributions which are invariant measures of diffusions processes. This class contains the most common continuous probability distributions, including the Gaussian, Gamma, Beta, Pareto, uniform, Student or log-normal distributions, among others. See also \cite{APP}, \cite{EV} for other attempts to extend the theory to more general target distribution via Malliavin calculus. Now, our purpose is to give necessary and sufficient conditions for the convergence of a sequence of random variables (regular enough in the Malliavin sense) to such target distributions. We obtain several results based on the Malliavin derivatives of the sequence and of the diffusion coefficients associated to the target distribution.
Precisely speaking, we prove the equivalence between the convergence in total variation of random variables under a suitable condition and the convergence of a value consisting of the diffusion coefficients and the Malliavin derivatives.
The value is associated to the so-called Stein factor, and the result also characterizes the convergence of the Stein factor.
Several situations when this theory can be applied (meaning that the diffusion coefficients have an explicit expression) are presented. Then we treat the case of the convergence in law of special sequences of random variables that belong to a Wiener chaos of fixed order and we characterize their convergence to target distributions that are different from the Gaussian and Gamma laws. For example, we obtain necessary and sufficient conditions in the case when the limit is a product normal distribution (the product of two independent normal random variables) or  the sum of a normal and an independent random variable.  

We will also focus our analysis on the particular case when the squared diffusion coefficient  is a polynomial of (at most) second degree. Several common probability laws are contained in this class. In this situation the necessary and sufficient conditions for the convergence of a sequence of multiple integrals to the target distribution can be analyzed in details. We actually show that, among the distributions whose associated  squared diffusion coefficient is a polynomial of second degree with some restrictions on its coefficients, the only possible limits of sequences of multiple integrals are the Gaussian and the Gamma laws. In particular, we show that a sequence of multiple Wiener-It\^o integrals cannot converge toward a beta or uniform distribution. We retrieve the  standard Fourth Moment Theorem (Theorem \ref{th:4momO}) and its version for the Gamma law as  particular cases.

We organized our paper as follows.  Section 2 contains some preliminaries on the Malliavin calculus. In Section 3 we recall and extend several results in \cite{KuTu} concerning the characterization of the random variables whose probability distribution is the invariant measure of a diffusion process. In Section 4 we give 
 necessary and sufficient conditions for the convergence of a sequence of random variables (regular enough in the Malliavin sense), while Section 5 and 6 treat the convergence in distribution of sequence of multiple stochastic integrals.

\section{Preliminary: Wiener-Chaos and Malliavin derivatives}
Here we describe the elements from stochastic analysis that we will need in the paper. Consider $H$ a real separable Hilbert space and $(W(h), h \in H)$ an isonormal Gaussian process on a probability space $(\Omega, {\cal{A}}, P)$, which is a centered Gaussian family of random variables such that ${\bf E}\left[ W(\varphi) W(\psi) \right]  = \langle\varphi, \psi\rangle_{H}$. Denote by  $I_{n}$ the multiple stochastic integral with respect to
$B$ (see \cite{N}). This mapping $I_{n}$ is actually an isometry between the Hilbert space $H^{\odot n}$(symmetric tensor product) equipped with the scaled norm $\frac{1}{\sqrt{n!}}\Vert\cdot\Vert_{H^{\otimes n}}$ and the Wiener chaos of order $n$ which is defined as the closed linear span of the random variables $H_{n}(W(h))$ where $h \in H, \|h\|_{H}=1$ and $H_{n}$ is the Hermite polynomial of degree $n \in {\mathbb N}$
\begin{equation*}
H_{n}(x)=\frac{(-1)^{n}}{n!} \exp \left( \frac{x^{2}}{2} \right)
\frac{d^{n}}{dx^{n}}\left( \exp \left( -\frac{x^{2}}{2}\right)
\right), \hskip0.5cm x\in \mathbb{R}.
\end{equation*}
The isometry of multiple integrals can be written as follows: for $m,n$ positive integers,
\begin{eqnarray}
\mathbf{E}\left(I_{n}(f) I_{m}(g) \right) &=& n! \langle \tilde{f},\tilde{g}\rangle _{H^{\otimes n}}\quad \mbox{if } m=n,\nonumber \\
\mathbf{E}\left(I_{n}(f) I_{m}(g) \right) &= & 0\quad \mbox{if } m\not=n.\label{iso}
\end{eqnarray}
It also holds that
\begin{equation*}
I_{n}(f) = I_{n}\big( \tilde{f}\big)
\end{equation*}
where $\tilde{f} $ denotes the symmetrization of $f$ defined by the formula $\tilde{f}%
(x_{1}, \ldots , x_{n}) =\frac{1}{n!} \sum_{\sigma \in {\cal S}_{n}}
f(x_{\sigma (1) }, \ldots , x_{\sigma (n) } ) $.
\\\\
We recall that any square integrable random variable which is measurable with respect to the $\sigma$-algebra generated by $W$ can be expanded into an orthogonal sum of multiple stochastic integrals
\begin{equation}
\label{sum1} F=\sum_{n=0}^\infty I_{n}(f_{n})
\end{equation}
where $f_{n}\in H^{\odot n}$ are (uniquely determined)
symmetric functions and $I_{0}(f_{0})=\mathbf{E}\left[  F\right]$.
\\\\
Let $L$ be the Ornstein-Uhlenbeck operator
\begin{equation*}
LF=-\sum_{n\geq 0} nI_{n}(f_{n})
\end{equation*}
if $F$ is given by (\ref{sum1}) and it is such that $\sum_{n=1}^{\infty} n^{2}n! \Vert f_{n} \Vert ^{2} _{{\cal{H}}^{\otimes n}}<\infty$.
\\\\
For $p>1$ and $\alpha \in \mathbb{R}$ we introduce the Sobolev-Watanabe space $\mathbb{D}^{\alpha ,p }$  as the closure of
the set of polynomial random variables with respect to the norm
\begin{equation*}
\Vert F\Vert _{\alpha , p} =\Vert (I -L) ^{\frac{\alpha }{2}} F \Vert_{L^{p} (\Omega )}
\end{equation*}
where $I$ represents the identity. We denote by $D$  the Malliavin  derivative operator that acts on smooth functions of the form $F=g(W(h_1), \dots , W(h_n))$ ($g$ is a smooth function with compact support and $h_i \in H$)
\begin{equation*}
DF=\sum_{i=1}^{n}\frac{\partial g}{\partial x_{i}}(W(h_1), \ldots , W(h_n)) h_{i}.
\end{equation*}
The operator $D$ is continuous from $\mathbb{D}^{\alpha , p} $ into $\mathbb{D} ^{\alpha -1, p} \left( H\right).$

We will intensively use the product formula for multiple integrals.
It is well-known that for $f\in H^{\odot n}$ and $g\in H^{\odot m}$
\begin{equation}\label{eq:Nualart2}
I_n(f)I_m(g)= \sum _{r=0}^{n\wedge m} r! \left( \begin{array}{c} n\\r\end{array}\right) \left( \begin{array}{c} m\\r\end{array}\right) I_{m+n-2r}(f\otimes _r g)
\end{equation}
where $f\otimes _r g$ means the $r$-contraction of $f$ and $g$ (see e.g. Section 1.1.2 in \cite{N}).


We recall the expression of the third and fourth moment of a random variable in a fixed Wiener chaos. These formulas play an important role in our proofs.
\begin{lemma}\label{8i1}
Let $F= I_{n}(f)$ with $n\in {\mathbb N}$ and $f\in H ^{\odot n}$. Then
\begin{equation}
\label{m3}
{\bf E} [F^3]= \frac{ (n!)^3}{[\left( n/2\right) ! ]^{3}}\langle f, f \tilde{ \otimes } _{n/2} f \rangle 1_{ \{ \text{n is even} \} }
\end{equation}
and
\begin{eqnarray}
{\bf E}[F^4]&=& 3 {\bf E}[F^2] ^{2}+ 3n \sum_{p=1} ^{n-1} (p-1) ! \left( \begin{array}{c} n-1\\ p-1\end{array}\right)^{2} p! \left( \begin{array}{c} n\\ p\end{array}\right) ^{2} (2n-2p)! \Vert f_{m} \tilde{ \otimes }_{p} f_{m} \Vert ^{2}\nonumber\\
&=&3 {\bf E}[I_{n}(f) ^2] ^{2}+ n! ^{2} \sum_{r=1} ^{n-1} (C_{n}^{r}) ^{2} \left[ \Vert f\otimes _{r} f \Vert ^{2} + C_{2n-2r}^{n-r} \Vert f\tilde{\otimes} _{r} f\Vert ^{2} \right] \label{m4}
\end{eqnarray}

\end{lemma}
\begin{proof}
We refer to \cite{NoPe4}, proof of Theorem 1.2 for the first two relations and to \cite{NPbook}, formula (5.2.6) for the last equality.
\end{proof}

The  Fourth Moment Theorem states as follows.

\begin{theorem}{\rm (\cite{NuPe} and  \cite{NuOr})}\label{th:4momO}
Fix $n \in {\mathbb N}$. Consider a sequence $(F_k=I_n(f_k))_{k \in {\mathbb N}}$ of square integrable random variables in the $n$-th Wiener chaos.
Assume that
\begin{equation}
\lim _{k\rightarrow \infty} \mathbf{E}[F_k^2] = \lim _{k\rightarrow \infty} \| f_k\| ^2_{H^{\odot n}} =1.
\end{equation}
Then, the following statements are equivalent.
\begin{enumerate}
\item \label{4momO1} The sequence of random variables $(F_k=I_n(f_k))_{ k\geq 1}$ converges to the standard normal law in distribution as $k\rightarrow \infty$.
\item \label{4momO2} $\lim _{k\rightarrow \infty}\mathbf{E}[F_k^4]=3$.
\item \label{4momO3} $\lim _{k\rightarrow \infty}\| f_k\otimes _l f_k\| _{H^{\otimes 2(n-l)}} =0$ for $l=1,2,\dots ,n-1$.
\item \label{4momO4} $\| DF_k\| _H^2$ converges to $n$ in $L^2(\Omega)$ as $k\rightarrow \infty$.
\end{enumerate}
\end{theorem}

We consider the general version of this theorem below.

\section{General versions of Stein's method and Stein's bound}\label{sec:review}

In order to discuss the general version of the Fourth Moment Theorem, we review Stein's method and Stein's bound obtained in \cite{KuTu} with small extension. Let us briefly recall the context in \cite{KuTu}. 
Let $S$ be the interval $(l,u) $ ($-\infty \leq l<u\leq \infty$) and $\mu$ be a probability measure on $S$ with a density function $p$ which is continuous,  strictly positive on $S$, and admits finite variance.
Consider a continuous function $b$ on $S$ such that there exists $k \in (l,u)$ such that $b(x)>0$  for $x \in (l,k)$ and $b(x)<0$ for $x\in (k,u)$, $b \in L^1(\mu)$, $bp$ is bounded on $S$ and
\[
\int_{l}^{u} b(x)p(x)dx =0.
\]
Define
\begin{equation}\label{n1}
a(x):= \frac{2 \int_{l}^{x} b(y) p(y) dy}{p(x)},\quad x\in S.
\end{equation}
Then, the stochastic differential equation:
\[
dX_{t}= b(X_{t}) dt + \sqrt{a(X_{t}) } dW_{t}, \hskip0.5cm t\geq 0
\]
has a unique Markovian weak solution, ergodic with invariant density $p$. See Theorem 2.4 in \cite{BSS}.

Based on this fact, it is possible to define a so-called Stein's equation for a given function $f\in L^{1}(\mu)$. 
In Section 3 of \cite{KuTu}, we have considered only the case that $f\in C_0(S)$ (the set of continuous functions on $S$ vanishing at the boundary of $S$). However, it is easy to see that the argument is valid  even if $f \in L^1(\mu )$, as follows.

For $f \in L^1(\mu )$, let $m_f:= \int_{l}^{u} f(x) \mu (dx)$ and define $\tilde g_f$ by, for every $x\in S$,
\begin{equation}\label{eq:defgf}
\tilde g_f(x):= \frac{2}{a(x)p(x)}\int _l ^x ( f(y)-m_f)p(y) dy.
\end{equation}
Then, by Proposition 1 in Section 3.2 of \cite{KuTu} we have
\[
\tilde g_f(x) = \int _l ^x \frac{2(f(y)-m_f)}{a(y)}\exp \left( -\int _y^x \frac{2b(z)}{a(z)}dz\right) dy, \hskip0.5cm x\in S.
\]
The function $g_f(x):=\int _0^x \tilde g_f(y)dy$ satisfies that $f-m_f=Ag_f$ ($A$ is the generator of the diffusion $(X_{t}) _{t\geq 0}$) $\mu$-almost everywhere, and
\begin{equation}\label{eq:Stein-general}
f(x) - \mathbf{E}[f(X)] = \frac12 a(x)\tilde g_f'(x) + b(x)\tilde g_f(x), \quad \mu \mbox{-a.e.}\ x
\end{equation}
where $X$ is a random variable with its law $\mu$.
The equation (\ref{eq:Stein-general}) is a generalized version of Stein's equation.

\begin{remark}\rm
\begin{enumerate}
\item If $f\in L^1(\mu) \cap C(S)$ (where $C(S)$ denotes the class of continuous functions on $S$), (\ref{eq:Stein-general}) holds for all $x \in S$.
\item Since $\mu$ has the density function $p$, (\ref{eq:Stein-general}) follows almost everywhere with respect to the Lebesgue measure.
\end{enumerate}
\end{remark}

Similarly to the original Stein's equation, (\ref{eq:Stein-general}) characterizes the distribution of $X$ as follows. This result will play a crucial role in the proofs of the main results in the next sections. 

\begin{theorem}\label{char2}
Assume that $\int _S a(x)\mu (dx) <\infty$.
Let $Y$ be a $S$-valued random variable such that $\mathbf{E}[|b(Y)|] <\infty$.
Then, the distribution of $Y$ coincides with $\mu$ if and only if
\begin{equation}\label{eq:char2}
\mathbf{E}\left[ \frac12 a(Y)h'(Y) + b(Y)h(Y)\right] =0
\end{equation}
for every $h\in C^1(S)$ such that $\mathbf{E}[|b(Y)h(Y)|] <\infty$ and $\mathbf{E}[| a(Y)h'(Y)|] <\infty$.
\end{theorem}

\begin{proof}
Assume that the distribution of $Y$ is $\mu$.
Let $h\in C_b^1(S)$ such that $\mathbf{E}[|b(Y)h(Y)|] <\infty$ and $\mathbf{E}[| a(Y)h'(Y)|] <\infty$, and
\[
f(x):= \frac 1{2p(x)}\frac d{dx}[a(x)p(x)h(x)], \quad x\in S.
\]
Here, note that $ap\in C^1(S)$ follows by the definition of $a$.
Since
\begin{align*}
\frac d{dx} [a(x)p(x)h(x)] &= [a(x)p(x)]' h(x) + a(x)p(x)h'(x)\\
&= 2b(x)p(x)h(x) + a(x)p(x)h'(x),
\end{align*}
we have
\begin{align*}
\int _S |f(x)| \mu (dx)
& = \frac 12 \int _S \left| \frac d{dx}[a(x)p(x)h(x)] \right| dx\\
& \leq \int _S |b(x)h(x)|p(x)dx + \frac 12 \int _S a(x)p(x)|h'(x)|dx \\
& = \mathbf{E}[|b(X)h(X)|] + \frac 12 \mathbf{E}[a(X)|h'(X)|].
\end{align*}
Hence, $f\in L^1(\mu )$.
Define $\tilde g_f$ as in (\ref{eq:defgf}).
Then, we have $h=\tilde g_f$ by explicit calculation.
Therefore, in view of (\ref{eq:Stein-general}), we obtain (\ref{eq:char2})  in the case that $h$ is bounded.
The case that $h$ is not bounded is obtained by approximation.

Next we show that the distribution of $Y$ is $\mu$ if (\ref{eq:char2}) holds.
Let $f\in C_K(S)$, where $C_K(S)$ is the total set of continuous functions on $S$ with compact support.
Since $f(x)=0$ for $x$ sufficiently near to $u$, there exists $u' \in (l,u)$ such that
\[
\int _l^x [f(y)-m_f]p(y)dy= m_f \left( 1- \int _l ^x p(y) dy \right) = m_f \int _x^u p(y)dy ,\quad x\in [u',u)
\]
where $m_f= \int _l^u f(x)p(x)dx$.
On the other hand, for sufficiently small $\varepsilon >0$ the assumption on $b$ implies that there exists $u''\in (l,u)$ such that
\[
b(x) < -\varepsilon \ \mbox{and} \ \int _x^u b(y)p(y) dy < 0
\]
for $x \in [u'',u)$.
Let $\tilde u := \max\{u',u''\}$.
Hence, by (\ref{n1}) we have for $x\in [ \tilde u, u)$
\begin{align*}
|\tilde g_f (x)| & = \frac{2}{a(x)p(x)}\left| \int _l ^x ( f(y)-m_f)p(y) dy \right|\\
& = \frac{|m_f| \int _x ^u p(y) dy }{\left| \int _l^x b(y)p(y)dy \right|} = \frac{|m_f| \int _x^u p(y) dy}{\left| \int _x^u b(y)p(y)dy \right| }\\
&\leq \frac{|m_f|}{\varepsilon}.
\end{align*}
Similarly, there exists $\tilde l \in (l,u)$ such that
\[
|\tilde g_f (x)| \leq \frac{|m_f|}{\varepsilon} ,\quad x\in (l,\tilde l] .
\]
These estimates and the continuity and positivity of $a(x)p(x)$ in $S$ imply that the function $\tilde g_f$ defined by (\ref{eq:defgf}) is bounded and satisfies $\mathbf{E}[|b(Y)\tilde g_f(Y)|] <\infty$.
The finiteness of $\mathbf{E}[| a(Y)\tilde g_f' (Y)|]$ is obtained from (\ref{eq:Stein-general}).
By (\ref{eq:Stein-general}) and (\ref{eq:char2}) we have
\[
\mathbf{E}[f(Y)]-\mathbf{E}[f(X)] = \mathbf{E}\left[ \frac12 a(Y)\tilde g_f'(Y) + b(Y)\tilde g_f(Y)\right] =0.
\]
Since this equality holds for any $f\in C_K(S)$, $X$ and $Y$ have the same distribution.
\end{proof}

An alternative characterization of the random variables $Y$ with distribution $\mu$ has been Theorem 2 in Section 3.2 of \cite{KuTu}. It involves operators from Malliavin calculus and a conditional expectation given the $\sigma$-field generated by $Y$. The same conditional expectation will appear in the statement of the main result in the next section.

\begin{theorem}\label{char}
Consider a random variable $Y \in \mathbb{D} ^{1,2}$ with its values on $S$ which satisfies that $b(Y)\in L^2(\Omega)$. Then, $Y$ has probability distribution $\mu$ if and only if $\mathbf{E}[b(Y)]=0$ and
\[
\mathbf{E}\left[ \left. \frac12 a(Y) + \langle D(-L)^{-1} b(Y), DY \rangle _H \right| Y\right] =0.
\]
\end{theorem}

Stein's bounds below for the distance between the law of an arbitrary random variable and the measure $\mu$ are based on the Stein equation (\ref{eq:Stein-general}) and the bounds of $\tilde g _f$ and $\tilde g_f'$ are obtained in Section 3.2 of \cite{KuTu}.
Now we state the result in \cite{KuTu} for later use.

\begin{ass}\label{ass}
\begin{enumerate}
\item
If $u<\infty$, assume that there exists $u' \in (l,u)$ such that $b$ is non-decreasing and Lipschitz continuous on $[u',u)$ and $\liminf_{x\rightarrow u}a(x)/(u-x) >0$.
If $u=\infty$, assume that there exists $u' \in (l,u)$ such that $b$ is non-decreasing on $[u',u)$ and $\liminf_{x\rightarrow u}a(x)>0$.
\item If $l>-\infty$, assume that there exists $l' \in (l,u)$ such that $b$ is non-increasing and Lipschitz continuous on $(l,l']$ and $\liminf_{x\rightarrow l}a(x)/(x-l) >0$.
If $l=-\infty$, assume that there exists $l' \in (l,u)$ such that $b$ is non-decreasing on $(l,l']$ and $\liminf_{x\rightarrow l}a(x)>0$.
\end{enumerate}
\end{ass}

\begin{theorem}\label{th:bound}
\begin{enumerate}

\item 
Let $d$ be the Fortet-Mourier distance.
Assume that there exist $l', u' \in (l,u)$ such that $b$ is non-increasing on $(l,l')$ and $(u',u)$.
Then we have for $S$-valued random variable $Y\in {\mathbb D}^{1,2}$
\begin{align*}
d({\mathcal L}(Y), \mu) & \leq C \mathbf{E}\left[ \left. \mathbf{E}\left[ \left| \frac12 a(Y) + \langle D(-L)^{-1}\left\{ b(Y)-\mathbf{E}[b(Y)]\right\} , DY \rangle _H \right| Y\right] \right| \right] \\
&\quad + C|\mathbf{E}\left[ b(Y)\right] |, \quad 
\end{align*}
where $C$ is a positive constant independent of $Y$ and ${\mathcal L}(Y)$ is the law of $Y$.

\item 
Let $d$ be the Kolmogorov distance or the total variation distance.
Under Assumption \ref{ass}, we have for $S$-valued random variable $Y\in {\mathbb D}^{1,2}$
\begin{align*}
d({\mathcal L}(Y), \mu)  &\leq C \mathbf{E}\left[ \left| \mathbf{E}\left[ \left.\frac12 a(Y) + \langle D(-L)^{-1}\left\{ b(Y)-\mathbf{E}[b(Y)]\right\} , DY \rangle _H \right| Y \right] \right| \right]\\
&\quad + C|\mathbf{E}\left[ b(Y)\right] |,
\end{align*}
where $C$ is a positive constant independent of $Y$ and ${\mathcal L}(Y)$ is the law of $Y$.
\end{enumerate}
\end{theorem}

We remark that Assumption \ref{ass} is designed for the estimate of the Kolmogorov distance or the total variation distance, while the estimate of the Fortet-Mourier distance is obtained under a simple assumption.
This difference of the assumptions comes from the estimate of $\tilde g_f'$.
We also remark that the cases that $\mu$ has the normal distribution and $\mu$ has the Gamma distribution satisfy the Assumption \ref{ass} under suitable choices of $a$ and $b$.
See Section 3.2 of \cite{KuTu} for the details.

The Fourth Moment Theorem tells us that the Stein's bound is sharp for multiple integrals in the case of Gaussian law, meaning that  a sequence of multiple stochastic integrals convergences to a Gaussian distribution if and only if the right-hand side of the Stein's bound vanishes. The purpose of the last three sections of our paper is to give a necessary and sufficient condition for  sequences of random variables to converge to a measure $\mu$ as described above.

\section{Necessary and sufficient conditions for the convergence to the invariant measure of a diffusion}\label{sec:NScond}

The purpose of this section is to provide necessary and sufficient conditions for the convergence of the sequence $ (F_{m}) _{m \in {\mathbb N}}$ to the invariant measure of a diffusion $\mu$ (as described in Section 3). 
We aim to give such a characterization in terms of the squared diffusion coefficient $a$ and of the Malliavin derivatives of $F_{m}$.  The main result in \cite{KuTu} (see also Section 3) implies that a sufficient condition for the convergence in distribution of $(F_{m}) _{m\in {\mathbb N}}$ in a certain class to $\mu $ is that ${\bf E}\left[ \left. \frac{1}{2} a(F_{m})- \langle DF_{m}, D(-L) ^{-1} \{ F_{m} -b(F_m)\} \rangle \right| F_m \right] $ converges in $L^1(\Omega)$ to zero as $m\to \infty$ and ${\bf E}[b(F_m)]$ converges to zero as $m\rightarrow \infty$.  But it will follow from the results presented below in this section that this condition is sometimes too strong and it is not a necessary condition for the convergence of $(F_{m})_{m\geq 1}$ in distribution to $\mu$.  Actually, we will consider the convergence in total variation which is strictly stronger than the convergence in distribution. Recall that the total variation distance between the law of two random variables $X$ and $Y$ is defined by
\begin{equation*}
d_{TV} ({\mathcal L}(X) ,{\mathcal L}(Y)) = \sup _{A} \left| P(X\in A)- P(Y\in A) \right|
\end{equation*}
where the supremum is taken over all Borel sets $A\subset \mathbb{R}$.  We also have (see e.g. Appendix C in \cite{NPbook})
\begin{equation*}
d_{TV} ({\mathcal L}(X) ,{\mathcal L}(Y)) = \frac{1}{2}\sup _{h}\left| \mathbf{E}\left[ h(X) \right] -\mathbf{E}\left[ h(Y) \right] \right|
\end{equation*}
where the supremum is considered over all Borel measurable functions $h$ with $\Vert h\Vert _{\infty} \leq 1$.
Let us start with the following result which is connected with Theorem \ref{char}.

Set $S=(l,u)$, $\mu$, $X$, $p$, $b$, $a$ and $k$ as in Section \ref{sec:review}.
Recall that $\int _S |b(x)|\mu (dx) <\infty$.
We assume that $\int _S a(x)\mu (dx) < \infty$.
Denote the Lebesgue measure on ${\mathbb R}$ by $dx$.
Additionally we consider the function $\phi$ on $S$ given by
\[
\phi (x):= \frac 12 a(x) + (|k|+|x|)|b(x)|
\]
where $k$ is an element in $S$ which appeared in the assumption of $b$.

\begin{theorem}\label{thm:equi}
Let $Y$ be an $S$-valued random variable in ${\mathbb D}^{1,2}$ and assume that the distribution of $Y$ is absolutely continuous with respect to the Lebesgue measure.
Then, for every $p>1$ we have
\begin{align*}
&{\bf E} \left[ \left| {\bf E} \left[ \left. \frac12 a(Y) + \langle D(-L)^{-1}\left\{ b(Y)-\mathbf{E}[b(Y)]\right\} , DY \rangle _H \right| Y \right] \right|  \right] \\
&\leq  \left\{ 1+ (1 + {\bf E} \left[ |b(Y)|^p\right] + {\bf E} \left[ |b(X)|^p\right]) {\bf E} \left[ |Y -k|\right] +  {\bf E} \left[ \phi (Y)^p \right] + {\bf E} \left[ \phi (X)^p \right] \right\} \\
&\quad \times d_{\rm TV}({\mathcal L}(Y), \mu) ^{1-1/p}.
\end{align*}
\end{theorem}

\begin{proof}
First we show that for $M>0$
\begin{equation}\label{eq:thmequi00} \begin{array}{l}
\displaystyle {\bf E} \left[ \left| {\bf E} \left[ \left. \frac12 a(Y) + \langle D(-L)^{-1}\left\{ b(Y)-\mathbf{E}[b(Y)]\right\} , DY \rangle _H \right| Y \right] \right|  \right] \\
\displaystyle \leq  (M+ {\bf E} \left[ |Y -k|\right]) d_{\rm TV}({\mathcal L}(Y), \mu) \\
\displaystyle \quad + {\bf E} \left[ \phi (Y); \phi (Y) > M \right] +  {\bf E} \left[ \phi (X); \phi (X) > M\right] \\
\displaystyle \quad + {\bf E} \left[ |Y -k|\right] ( {\bf E}\left[ |b(Y)| ; |b(Y)|> M \right] + {\bf E}\left[| b(X)| ; |b(X)| > M\right] ).
\end{array} \end{equation}

Let $h \in L^\infty (S,dx)$ (the space of essentially bounded functions on $S$ with respect to $dx$).
For $x \in S$ define 
\[
f_h (x):= \frac 12 a(x)h(x) + b(x) \int _k^x h(y) dy .
\]
We remark that the existence of the density function of $X$ implies $|h(X)| \leq \| h\| _\infty$ almost surely.
Since ${\bf E}[|a(X)h(X)|]<\infty$ and 
\begin{align*}
{\bf E}\left[ \left| b(X) \int _k^X h(y) dy\right| \right] &\leq \| h\| _\infty {\bf E}\left[ \left| b(X) X\right| \right] + |k| \| h\| _\infty {\bf E}\left[ \left| b(X) \right| \right] \\
&<\infty ,
\end{align*}
Theorem \ref{char2} implies $E[f_h(X)]=0$.
Hence, by integration by parts formula, we have
\begin{align*}
&{\bf E} \left[ f_h(Y)\right] -  {\bf E} \left[ f_h(X)\right] \\
&= {\bf E} \left[  \frac12 a(Y)h(Y) + b(Y)\int _k^Y h(y)dy \right]\\
&= {\bf E} \left[  \frac12 a(Y)h(Y) + (b(Y)-{\bf E}[b(Y)])\int _k^Y h(y)dy \right] + {\bf E}[b(Y)] {\bf E} \left[  \int _k^Y h(y)dy \right]\\
&= {\bf E} \left[  \frac12 a(Y)h(Y) + \left[ \delta D (-L)^{-1} (b(Y)-{\bf E}[b(Y)]) \right] \int _k^Y h(y)dy \right] \\
&\quad + {\bf E}[b(Y)] {\bf E} \left[  \int _k^Y h(y)dy \right]\\
&= {\bf E} \left[  \frac12 a(Y)h(Y) + \left\langle D (-L)^{-1} (b(Y)-{\bf E}[b(Y)]), D\int _k^Y h(y)dy \right\rangle _H \right] \\
&\quad + {\bf E}[b(Y)] {\bf E} \left[  \int _k^Y h(y)dy \right]\\
&= {\bf E} \left[ h(Y) \left( \frac12 a(Y) + \left\langle D (-L)^{-1} (b(Y)-{\bf E}[b(Y)]), DY \right\rangle _H \right) \right] \\
&\quad + {\bf E}[b(Y)] {\bf E} \left[  \int _k^Y h(y)dy \right].
\end{align*}
Since the duality between the $L^1(S,dx)$ and the $L^\infty(S,dx)$ yields
\begin{align*}
& {\bf E} \left[ \left| {\bf E} \left[ \left. \frac12 a(Y) + \langle D(-L)^{-1}\left\{ b(Y)-\mathbf{E}[b(Y)]\right\} , DY \rangle _H \right| Y \right] \right|  \right] \\
&=\sup_{\| h\| _\infty \leq 1}\left| {\bf E} \left[ \left( \frac12 a(Y) + \langle D(-L)^{-1}\left\{ b(Y)-\mathbf{E}[b(Y)]\right\} , DY \rangle _H \right) h(Y) \right] \right| ,
\end{align*}
we obtain
\begin{equation}\label{eq:thmequi01}
\begin{array}{l}
\displaystyle {\bf E} \left[ \left| {\bf E} \left[ \left. \frac12 a(Y) + \langle D(-L)^{-1}\left\{ b(Y)-\mathbf{E}[b(Y)]\right\} , DY \rangle _H \right| Y \right] \right|  \right]  \\[5mm]
\displaystyle \leq \sup_{\| h\| _\infty \leq 1} \left| {\bf E} \left[ f_h (Y)\right] -  {\bf E} \left[ f_h (X)\right] \right| +  \left| \mathbf{E}\left[ b(Y)\right] \right| {\bf E} \left[ |Y -k|\right] .
\end{array}
\end{equation}
Now we consider the estimate of $f_h$.
For $x\in S$, it holds that
\begin{align*}
|f_h(x)|
& \leq \left| b(x)\int _k^x h(y) dy \right|  + \frac 12 a(x)|h(x)| \\
& \leq \| h\| _\infty  (|k|+|x|)|b(x)| + \frac 12 a(x)|h(x)| .
\end{align*}
Since the existence of the density function of $X$ and $Y$ implies $|h(X)| \leq \| h\| _\infty$ and $|h(Y)| \leq \| h\| _\infty$ almost surely, we have
\begin{align*}
& |f_h(X)| \leq \| h\| _\infty \left\{ (|k|+|X|)|b(X)| + \frac 12 a(X) \right\} = \| h\| _\infty \phi (X), \\
& |f_h(Y)| \leq \| h\| _\infty \left\{ (|k|+|Y|)|b(Y)| + \frac 12 a(Y) \right\} = \| h\| _\infty \phi (Y)
\end{align*}
almost surely.
Hence,
\begin{align*}
&\left| {\bf E} \left[ f_h (Y)\right] -  {\bf E} \left[ f_h (X)\right] \right| \\[2mm]
&\leq \left| {\bf E} \left[ f_h (Y); \phi (Y) \leq M \right] -  {\bf E} \left[ f_h (X); \phi (X) \leq M\right] \right| \\[2mm]
&\quad + \left| {\bf E} \left[ f_h (Y); \phi (Y) > M \right] -  {\bf E} \left[ f_h (X); \phi (X) > M\right] \right| \\[2mm]
&\leq \| h\| _\infty M d_{\rm TV}({\mathcal L}(Y), \mu) \\[2mm]
&\quad + \| h\| _\infty \left( {\bf E} \left[ \phi (Y); \phi (Y) > M \right] +  {\bf E} \left[ \phi (X); \phi (X) > M\right] \right) .
\end{align*}
This inequality and (\ref{eq:thmequi01}) yield
\begin{equation}\label{eq:thmequi02}
\begin{array}{l}
\displaystyle {\bf E} \left[ \left| {\bf E} \left[ \left. \frac12 a(Y) + \langle D(-L)^{-1}\left\{ b(Y)-\mathbf{E}[b(Y)]\right\} , DY \rangle _H \right| Y \right] \right|  \right]  \\[5mm]
\displaystyle \leq M d_{\rm TV}({\mathcal L}(Y), \mu) +  {\bf E} \left[ \phi (Y); \phi (Y) > M \right] +  {\bf E} \left[ \phi (X); \phi (X) > M\right] \\[3mm]
\displaystyle \quad +  \left| \mathbf{E}\left[ b(Y)\right] \right| {\bf E} \left[ |Y -k|\right] .
\end{array}
\end{equation}
Since ${\bf E}[b(X)]=0$, we have
\begin{align*}
& \left| \mathbf{E}\left[ b(Y)\right] \right| \\
& =  \left| {\bf E}\left[ b(Y)\right] - {\bf E}\left[ b(X)\right] \right| \\
& \leq  \left| {\bf E}\left[ b(Y); |b(Y)|\leq M \right] - {\bf E}\left[ b(X); |b(X)| \leq M\right] \right| \\
&\quad + {\bf E}\left[ |b(Y)| ; |b(Y)|> M \right] + {\bf E}\left[| b(X)| ; |b(X)| > M\right] \\
&\leq M d_{\rm TV}({\mathcal L}(Y), \mu) + {\bf E}\left[ |b(Y)| ; |b(Y)|> M \right] + {\bf E}\left[| b(X)| ; |b(X)| > M\right] .
\end{align*}
From this inequality and (\ref{eq:thmequi02}) we obtain (\ref{eq:thmequi00}).

It is easy to see
\begin{align*}
{\bf E} \left[ \phi (Y); \phi (Y) > M \right] 
&\leq M ^{1-p} {\bf E} \left[ \phi (Y)^p; \phi (Y) > M \right] \\
&\leq M ^{1-p} {\bf E} \left[ \phi (Y)^p \right] .
\end{align*}
Similarly we have
\begin{align*}
{\bf E} \left[ \phi (X); \phi (X) > M \right] &\leq M ^{1-p} {\bf E} \left[ \phi (X)^p \right] , \\
{\bf E} \left[ |b(Y)| ; |b(Y)| > M \right] &\leq M ^{1-p} {\bf E} \left[ |b(Y)|^p \right], \\
{\bf E} \left[ |b(X)| ; |b(X)| > M \right] &\leq M ^{1-p} {\bf E} \left[ |b(X)|^p \right] .
\end{align*}
Hence, (\ref{eq:thmequi00}) implies
\begin{align*}
&{\bf E} \left[ \left| {\bf E} \left[ \left. \frac12 a(Y) + \langle D(-L)^{-1}\left\{ b(Y)-\mathbf{E}[b(Y)]\right\} , DY \rangle _H \right| Y \right] \right|  \right] \\
&\leq (M + {\bf E} \left[ |Y -k|\right]) d_{\rm TV}({\mathcal L}(Y), \mu) \\[2mm]
&\quad + M^{1-p} ( {\bf E} \left[ |b(Y)|^p + |b(X)|^p \right] {\bf E} \left[ |Y-k| \right]+  {\bf E} \left[ \phi (Y)^p + \phi (X)^p \right]) .
\end{align*}
Letting $M= d_{\rm TV}({\mathcal L}(Y), \mu) ^{-1/p}$ and using the fact that $d_{\rm TV}({\mathcal L}(Y), \mu) \leq 1$, we obtain the desired estimate.
\end{proof}

Theorems \ref{th:bound} and \ref{thm:equi} give a necessary and sufficient condition of the convergence of in total variation.

\begin{theorem}\label{thm:conv}
Suppose Assumption \ref{ass}.
Let $F_m$ be a sequence of $S$-valued random variables in ${\mathbb D}^{1,2}$ and assume that the distribution of $F_m$ is absolutely continuous with respect to the Lebesgue measure for $m\in {\mathbb N}$.
Assume that there exists $\varepsilon >0$ such that
\begin{align*}
&{\bf E} \left[ | X| \right] + {\bf E} \left[ |b(X)|^{1+\varepsilon} \right] +  {\bf E} \left[ \phi (X)^{1+\varepsilon} \right] <\infty \\
&\sup _{m\in {\mathbb N}} \left( {\bf E} \left[ | F_m| \right] + {\bf E} \left[ |b(F_m)|^{1+\varepsilon} \right] + {\bf E} \left[ \phi (F_m)^{1+\varepsilon} \right] \right) <\infty .
\end{align*}
Then, the following statements are equivalent.
\begin{enumerate}
\item \label{thm:conv1} The distribution of $F_m$ converges to the distribution of $X$ in total variation.
\item \label{thm:conv2} $\displaystyle \lim _{m\rightarrow \infty} {\bf E} \left[ \left|\frac12 a(F_m) +  {\bf E} \left[ \left. \langle D(-L)^{-1}\left\{ b(F_m)-\mathbf{E}[b(F_m)]\right\} , DF_m \rangle _H \right| F_m \right] \right|  \right] =0$ and $\displaystyle \lim _{m \rightarrow \infty} E[b(F_m)] =0$.
\end{enumerate}
\end{theorem}

\begin{proof}
The implication of \ref{thm:conv1} to \ref{thm:conv2} follows from Theorem \ref{thm:equi}, the uniform integrability of $\{ b(F_m)\}$ and ${\bf E}[b(X)]=0$.
The converse implication follows from Theorem \ref{th:bound}.
\end{proof}

\begin{remark}
\begin{itemize}
\item Our Theorem \ref{thm:conv} is related to the result stated in Theorem 3.2 in \cite{APP}. In this reference, a criterium involving the conditional expectation for the convergence towards a linear combination of chi squared random variable is given. 
\item If the random variables $F_{m}$ from Theorem \ref{thm:conv} belong to a Wiener chaos of fixed orde and the limit distribution is not trivial, then the convergence in total variation is equivalent to the convergence in distribution (see \cite{NoPo}).
\item The result in Theorem \ref{thm:conv} indicates that the condition 
\[
\frac{1}{2} a(F_{m})- {\bf E}\left[ \left. \langle DF_{m}, D(-L) ^{-1} \left( b(F_{m}) - {\bf E}[b(F_m)] \right) \rangle _H \right| F_m \right] \to _{m} 0\ \mbox{in}\ L^1(\Omega)
\]
is sometimes maybe stronger than the convergence in law.
\end{itemize}
\end{remark}

From now on, for simplicity, we additionally assume that ${\bf E} [X] = \int _S x \mu (dx) =0$ and $b(x)=-x$.
In this case, $k=0 \in S$ and $\phi (x):= \frac 12 a(x) + |x|^2$.

\begin{corollary}\label{cor:conv}
Suppose Assumption \ref{ass}.
Let $F_m$ be a sequence of $S$-valued random variables in ${\mathbb D}^{1,2}$ and assume that the distribution of $F_m$ is absolutely continuous with respect to the Lebesgue measure for $m\in {\mathbb N}$.
Assume that there exists $\varepsilon >0$ such that
\[
{\bf E} \left[ \frac 12 a(X)^{1+\varepsilon} + |X|^{2+\varepsilon} \right] <\infty , \sup _{m\in {\mathbb N}} {\bf E} \left[ \frac 12 a(F_m)^{1+\varepsilon} + |F_m|^{2+\varepsilon} \right] <\infty .
\]
Then, the following statements are equivalent.
\begin{enumerate}
\item \label{cor:conv1} The distributions of $F_m$ converges to the distribution of $X$ in total variation.
\item \label{cor:conv2} $\displaystyle \lim _{m\rightarrow \infty} {\bf E} \left[ \left| \frac12 a(F_m) - {\bf E} \left[ \left. \langle D(-L)^{-1} F_m, DF_m \rangle _H \right| F_m \right] \right|  \right] =0$\\ and $\displaystyle \lim _{m \rightarrow \infty} E[F_m] =0$.
\end{enumerate}
\end{corollary}

\begin{remark}
The conditional expectation
\[
{\bf E} \left[ \left. \langle D(-L)^{-1} F, DF \rangle _H \right| F \right]
\]
is called the Stein factor associated with the random variable $F$, and Corollary \ref{cor:conv} implies that the convergence of the distributions in total variation is characterized by the Stein factor.
Recently, the Stein factor and the applications are studied as a hot topic, and many results have been obtained (see \cite{NPS} for details).
\end{remark}

We continue the consideration of the necessary and sufficient conditions of the convergence in distributions.
We additionally assume $\int _S a(x)^2 \mu (dx) < \infty$.

\begin{prop}\label{prop:prop24}
Consider a sequence $(F_m) _{m\geq 1 }$ of $S$-valued random variables in $\mathbb{D} ^{1,4}$, which satisfies that there exists $\varepsilon >0$ satisfying
\begin{equation}\label{eq:assprop24}
\sup _m {\bf E}\left[ a(F_m)^{2+\varepsilon} + |F_m|^{4+\varepsilon} \right] <\infty.
\end{equation}
Suppose that the distribution of $F_m$ converges to $\mu$ as $m\to \infty$ Then the condition that $\frac12 a(F_m) - \langle DF_{m}, D(-L) ^{-1} F_{m} \rangle _H$ converges to $0$ in $L^2(\Omega )$ as $m\rightarrow \infty$ is equivalent to the condition that $\frac 14 \mathbf{E}[a(F_m)^2] -  \mathbf{E}[\langle DF_{m}, D(-L) ^{-1} F_{m} \rangle _H^2 ]$ converges to $0$ as $m\rightarrow \infty$.
\end{prop}

\begin{proof}
Since distribution of $F_m$ converges to $\mu$, by Theorem \ref{char2} we have
\[
\lim _{m\rightarrow \infty }\mathbf{E}\left[ \frac12 a(F_m)h'(F_m) - F_m h(F_m)\right] =0
\]
for $h\in C^1(S)$ such that $\sup _m {\bf E}\left[ |h(F_m)|^{2+\varepsilon} + |h'(F_m)|^{2+\varepsilon} \right] <\infty$.
On the other hand,
\[
\mathbf{E}\left[ F_m h(F_m)\right] = \mathbf{E}\left[ h(F_m) (\delta D)(-L)^{-1} F_m\right] =  \mathbf{E}\left[ h'(F_m)\langle DF_{m}, D(-L) ^{-1} F_{m} \rangle _H \right] .
\]
Hence, we have
\[
\lim _{m\rightarrow \infty }\mathbf{E}\left[ \left( \frac12 a(F_m) - \langle DF_{m}, D(-L) ^{-1} F_{m} \rangle _H \right) h'(F_m)\right] =0
\]
for $h\in C^1(S)$ such that $\sup _m {\bf E}\left[ |h(F_m)|^{2+\varepsilon} + |h'(F_m)|^{2+\varepsilon} \right] <\infty$.
We have assumed (\ref{eq:assprop24}).
Hence, we can choose $h\in C^1(S)$ such that $h'(x) = a(x)$.
Thus, we obtain
\begin{equation}\label{eq:prop24}
\lim _{m\rightarrow \infty }\left( \frac12 \mathbf{E}\left[ a(F_m)^2\right] - \mathbf{E}\left[ a(F_m) \langle DF_{m}, D(-L) ^{-1} F_{m} \rangle _H \right] \right) =0.
\end{equation}
On the other hand,
\begin{align*}
& \mathbf{E}\left[ \left( \frac12 a(F_m) - \langle DF_{m}, D(-L) ^{-1} F_{m} \rangle _H \right) ^2 \right] \\
& = \mathbf{E}[\langle DF_{m}, D(-L) ^{-1} F_{m} \rangle _H^{2}]  - \frac 14 \mathbf{E}[a(F_m)^2] \\
&\hspace{2cm} + \frac12 \mathbf{E}\left[ a(F_m)^2\right] -  \mathbf{E}\left[ a(F_m) \langle DF_{m}, D(-L) ^{-1} F_{m} \rangle _H \right].
\end{align*}
Hence, applying (\ref{eq:prop24}), we obtain the required equivalence.
\end{proof}

The necessary and sufficient condition in Theorem \ref{thm:conv} for the convergence of the sequence $ (F_{m}) _{m\geq 1}$ to the distribution $\mu$ in total variation is that the sequence 
\[
{\bf E} \left[ \left. \frac12 a(F_m) + \langle D(-L)^{-1}\left\{ b(F_m)-\mathbf{E}[b(F_m)]\right\} , DF_m \rangle _H \right| F_m \right]
\]
converges to zero in $L^{1}(\Omega)$ as $m\to \infty$. Due to the appearance of the conditional expectation, this  condition is sometimes hard to be checked. Therefore, we give below in Theorem \ref{thm:4momG} an  alternative result which does not involve the conditional expectation.  

\begin{theorem}\label{thm:4momG}
Assume that there exists a random variable $G \in {\mathbb D}^{1,4}$ such that the distribution of $G$ is equal to $\mu$ and $\langle DG, D(-L)^{-1}G \rangle _H$ is measurable with respect to the $\sigma$-field generated by $G$.
Consider a sequence $(F_m) _{m\geq 1}$ of $S$-valued random variables in $\mathbb{D} ^{1,4}$ such that ${\bf E}[F_m]=0$ and (\ref{eq:assprop24}) is satisfied.
Then, the following statements are equivalent.
\begin{enumerate}
\item \label{4momG1} The vector valued random variable $(F_m, \langle DF_{m}, D(-L) ^{-1} F_{m}\rangle _H)$ converges to $(G, \langle DG, D(-L)^{-1} G\rangle _{H} )$ in distribution as $m\rightarrow \infty$, and $\{ a(F_m) ^2\}$ and $\langle DF_{m}, D(-L) ^{-1} F_{m}\rangle _H ^2$ are uniformly integrable.
\item \label{4momG4} $\frac12 a(F_m) - \langle DF_{m}, D(-L) ^{-1} F_{m}\rangle _H$ converges to $0$ in $L^2(\Omega)$ as $m\rightarrow \infty$.
\end{enumerate}
\end{theorem}

\begin{proof}
Assume $\ref{4momG4}$.
By Theorem \ref{th:bound} $F_m$ converges to $G$ in distribution.
Since $\langle DG, D(-L)^{-1}G \rangle _H$ is measurable with respect to the $\sigma$-field generated by $G$, by Theorem \ref{char} we obtain
\begin{equation}\label{eq:thm4momG}
\frac 12 a(G)= \langle DG, D(-L)^{-1}G \rangle _H
\end{equation}
almost surely.
Hence, by the convergence of $F_m$ to $G$ in distribution and $\ref{4momG4}$, we have for $h_1,h_2 \in C_b({\mathbb R})$
\begin{align*}
& \limsup _{m\rightarrow \infty} \left| \mathbf{E}\left[ h_1(F_m) h_2\left( \langle DF_{m}, D(-L) ^{-1} F_{m}\rangle _H \right) \right] \right. \\
&\quad \hspace{5cm} \left. - \mathbf{E}\left[ h_1(G) h_2\left( \langle DG, D(-L)^{-1}G\rangle _{H} \right) \right] \right| \\
& \leq \limsup _{m\rightarrow \infty} \left| \mathbf{E}\left[ h_1(F_m) h_2\left( \frac 12 a(F_m) \right) \right] - \mathbf{E}\left[ h_1(G) h_2\left( \frac 12 a(G) \right) \right] \right| \\
& \quad + \limsup _{m\rightarrow \infty} \left| \mathbf{E}\left[ h_1(F_m) \left\{ h_2\left( \frac 12 a(F_m) \right) - h_2\left( \langle DF_{m}, D(-L) ^{-1} F_{m}\rangle _H \right) \right\} \right] \right| \\
& =0.
\end{align*}
Thus, \ref{4momG1} is obtained.

Assume \ref{4momG1}.
Noting that (\ref{eq:thm4momG}) holds, by the assumption \ref{4momG1} we have
\begin{align*}
&\lim _{m\rightarrow \infty}\left( \frac 14 \mathbf{E}[a(F_m)^2] - {\bf E}\left[ \langle DF_{m}, D(-L) ^{-1} F_{m}\rangle _H^2 \right]\right) \\
&\hspace{2cm} = \frac 14 \mathbf{E}[a(G)^2] - \mathbf{E}\left[ \langle DG, D(-L)^{-1}G \rangle _H ^2\right] =0.
\end{align*}
Thus, by Proposition \ref{prop:prop24}, \ref{4momG4} is obtained.
\end{proof}

\begin{remark}
If we suppose the assumptions in both Corollary \ref{cor:conv} and Theorem \ref{thm:4momG}, then the equivalent conditions in Corollary \ref{cor:conv} and Theorem \ref{thm:4momG} will be also equivalent.
This fact  perhaps implies that generally the convergence of distributions to the invariant measures of diffusion processes would be complicated, while the fourth moment theorems with respect to Gaussian distributions and Gamma distributions are simple.
\end{remark}

In order to apply the above result, we need to know how to compute the squared diffusion coefficient for a given law and to check the measurability of $\langle DG, D(-L) ^{-1} G \rangle _H$ with respect to $G$. 

\begin{prop}\label{prop:mble}
\begin{enumerate}
\item \label{prop:mble1} Let $F = cW(h)$ where $c \in {\mathbb R}$ and $h\in H$ such that $\| h\| _H =1$.
Then,
\[
\langle D(-L)^{-1}F, DF\rangle _H = c^2.
\]
In particular, $\langle D(-L)^{-1}F, DF\rangle _H$ is measurable with respect to the $\sigma$-field generated by $F$.

\item \label{prop:mble1.5} Let $F = c(W(h)^2-1)$ where $c \in {\mathbb R}$ and $h\in H$ such that $\| h\| _H =1$.
Then,
\[
\langle D(-L)^{-1}(F-{\mathbf E}[F]), DF\rangle _H = 2cF + 2c^2.
\]
In particular, $\langle D(-L)^{-1}F, DF\rangle _H$ is measurable with respect to the $\sigma$-field generated by $F$.

\item \label{prop:mble2} Let $F=e^{cW(h)}$ where $c \in {\mathbb R}$ and $h\in H$ such that $\| h\| _H =1$.
Then,
\[
\langle D(-L)^{-1}(F-{\mathbf E}[F]), DF\rangle _H = c^2 F \int _0^1 F^v e^{\frac {c^2}2(1-v^2)} dv.
\]
In particular, $\langle D(-L)^{-1}(F-{\mathbf E}[F]), DF\rangle _H$ is measurable with respect to the $\sigma$-field generated by $F$.

\item \label{prop:mble3} Let $F= e^{c\sum _{k=1}^n W(h_k) ^2}$ where $n\in {\mathbb N}$, $c \in (-\infty , 1/2)$ and $h_1,h_2,\dots ,h_n \in H$ such that $\| h_k\| _H =1$ for $k=1,2,\dots ,n$, and $\langle h_k, h_l\rangle _H =0$ for $k,l =1,2,\dots ,n$.
Then,
\[
\langle D(-L)^{-1}(F-{\mathbf E}[F]), DF\rangle _H =   4c F(\log F) \int _0^1 \frac{vF^{ \frac{v^2}{1-2c(1-v^2)}}}{(1-2c(1-v^2))^{\frac n2 +1}}  dv.
\]
In particular, $\langle D(-L)^{-1}(F-{\mathbf E}[F]), DF\rangle _H$ is measurable with respect to the $\sigma$-field generated by $F$.

\item \label{prop:mble5} Let $F= W(h)^n -{\bf E}\left[ W(h)^n \right]$ where $n\in {\mathbb N}$ and $h\in H$ such that $\| h\| _H =1$.
Then,
\begin{align*}
&\langle D(-L)^{-1}F, DF\rangle _H \\
&= \frac{n^{2}}{2} \sum_{l=0} ^{\lfloor (n-1)/2 \rfloor} \frac{(n-1)! (2l-1)!!}{(2l)! \, (n-1-2l)!} \beta \left( \frac n2 -l,l+1\right) \left| F+{\bf E}\left[ W(h)^n\right] \right| ^{2-2(l+1)/n}.
\end{align*}
where $\lfloor (n-1)/2 \rfloor$ is the largest integer which is no larger than $(n-1)/2$.
In particular, $\langle D(-L)^{-1}F, DF\rangle _H$ is measurable with respect to the $\sigma$-field generated by $F$.
\end{enumerate}
\end{prop}

\begin{proof}
When $F$ is given as in \ref{prop:mble1}, ${\mathbf E}[F]=0$ and
\[
\langle D(-L)^{-1}F, DF\rangle _H = c^2 \langle DW(h), DW(h)\rangle _H = c^2.
\]
Hence, \ref{prop:mble1} holds.

When $F$ is given as in \ref{prop:mble1.5}, ${\mathbf E}[F]=0$ and
\begin{align*}
\langle D(-L)^{-1}F, DF\rangle _H &= \frac 12 c^2 \langle D(W(h)^2-1), D(W(h)^2-1)\rangle _H \\
&= 2c^2 W(h)^2 = 2cF + 2c^2.
\end{align*}
Hence, \ref{prop:mble1.5} holds.

Next we show the case that $F$ is given as in \ref{prop:mble2}.
In this case, we use the following formula proved in \cite{NV}:
if  $Y= f(N)-{\mathbf E}[f(N)]$ where $f\in C_b^1( \mathbb{R} ^{n}; \mathbb{R})$ with bounded derivatives and $N=(N_{1},..., N_{n}) $ is a Gaussian vector   with zero mean and covariance matrix $K=(K_{i,j}) _{i,j=1,..,n}$ then
\begin{equation}\label{nv}\begin{array}{l}
\displaystyle \langle D(-L) ^{-1} (Y-{\mathbf E}[Y]), DY\rangle _H \\
\displaystyle = \int_0 ^\infty e^{-u} du {\mathbf E}'\left[ \sum_{i,j=1} ^{n} K_{i,j} \frac{\partial f}{\partial x_{i}} (N) \frac{\partial f} {\partial x_{j}} (e^{-u}N+ \sqrt{1-e^{-2u}}N' )\right].
\end{array}\end{equation}
Here  $N'$ denotes an independent copy of $N$, and $N$ and $N'$ are defined on a product probability space
 $\left( \Omega \times \Omega ', {\cal{F}} \otimes {\cal{F}}, P\times P'\right)$ and ${\mathbf E}'$
 denotes the expectation with respect to the probability measure $P'$.

By (\ref{nv}), we have
\begin{align*}
&\langle D(-L)^{-1}(F-{\mathbf E}[F]), DF\rangle _H \\
&= \int _0^\infty e^{-u}du {\mathbf E}'\left[ c^2 e^{cW(h)} e^{c(e^{-u}W(h)+\sqrt{1-e^{-2u}}W'(h))}\right] \\
&= c^2 e^{cW(h)} \int _0^1 e^{cvW(h)} {\mathbf E}'\left[ e^{c\sqrt{1-v^2} W'(h)}\right] dv\\
&= c^2 F \int _0^1 F^v e^{\frac {c^2}2(1-v^2)} dv.
\end{align*}
Hence, \ref{prop:mble2} holds.

Next we show the case that $F$ is given as in \ref{prop:mble3}.
In this case, we use Lemma 1 in Section 4 of \cite{KuTu}.
The statement of the lemma is as follows.
Let $Z$ be a random variable with the standard normal distribution, $K>-\frac 12$, $C\in {\mathbb{R}}$ and $a\in (0,1)$.
Then,
\begin{align}
&\label{aux1} {\mathbf E}\left[ e^{ -K( C + \sqrt{1-a^{2}}Z  ) ^{2}} \right] = \frac{1}{ \sqrt{1+2K (1-a^{2})}} e^{- \frac{C^{2}K}{1+2K (1-a^{2}) }}\\
&\label{aux2} {\mathbf E}\left[ \left( C + \sqrt{1-a^{2}} Z\right) e^{ -K( C + \sqrt{1-a^{2}}Z  ) ^{2}} \right] =\frac{C}{\left( 1+ 2K (1-a^{2}) \right) ^{\frac{3}{2}}}e^{- \frac{C^{2}K}{1+2K (1-a^{2}) }}.
\end{align}

It is sufficient to show the case that $c\neq 0$.
By (\ref{nv}), we have
\begin{align*}
&\langle D(-L)^{-1}(F-{\mathbf E}[F]), DF\rangle _H \\
&= \int _0^\infty e^{-u}du {\mathbf E}'\left[ \sum _{i=1}^n 2cW(h_i) e^{c\sum _{j=1}^n W(h_j)^2} \right. \\
&\left. \phantom{\sum _{i=1}^n} \times 2c (e^{-u}W(h_i)+\sqrt{1-e^{-2u}}W'(h_i))e^{c\sum _{j=1}^n(e^{-u}W(h_j)+\sqrt{1-e^{-2u}}W'(h_j))^2}\right] \\
&= 4c^2 \int _0^\infty e^{-u}du  \sum _{i=1}^nW(h_i) e^{c\sum _{j=1}^n W(h_j)^2} \prod _{j\neq i} {\mathbf E}'\left[ e^{c(e^{-u}W(h_j)+\sqrt{1-e^{-2u}}W'(h_j))^2}\right]\\
&\quad \times {\mathbf E}'\left[ (e^{-u}W(h_i)+\sqrt{1-e^{-2u}}W'(h_i))e^{c(e^{-u}W(h_i)+\sqrt{1-e^{-2u}}W'(h_i))^2}\right] \\
&= 4c^2 \int _0^1 dv  \sum _{i=1}^nW(h_i) e^{c\sum _{j=1}^n W(h_j)^2} \left( \prod _{j\neq i} {\mathbf E}'\left[ e^{c(vW(h_j)+\sqrt{1-v^2}W'(h_j))^2}\right] \right) \\
&\quad \times {\mathbf E}'\left[ (vW(h_i)+\sqrt{1-v^2}W'(h_i))e^{c(vW(h_i)+\sqrt{1-v^2}W'(h_i))^2}\right].
\end{align*}
Applying (\ref{aux1}) and (\ref{aux2}) to this equation, we obtain
\begin{align*}
&\langle D(-L)^{-1}(F-{\mathbf E}[F]), DF\rangle _H \\
&= 4c^2 \int _0^1 dv  \sum _{i=1}^nW(h_i) e^{c\sum _{j=1}^n W(h_j)^2} \left( \prod _{j\neq i} \frac{\exp\left( \frac{c v^2 W(h_j)^2}{1-2c(1-v^2)}\right)}{\sqrt{1-2c(1-v^2)}} \right) \\
&\quad \times \frac{vW(h_i)\exp\left( \frac{cv^2 W(h_j)^2}{1-2c(1-v^2)}\right)}{(1-2c(1-v^2))^{\frac 32}} \\
&=  4c^2\left( \sum _{i=1}^n W(h_i)^2 \right) e^{c\sum _{j=1}^n W(h_j)^2} \int _0^1 \frac{v\exp \left( \frac{cv^2 \sum _{j=1}^n W(h_j)^2}{1-2c(1-v^2)}\right)}{(1-2c(1-v^2))^{\frac n2 +1}}  dv \\
&= 4c F(\log F) \int _0^1 \frac{vF^{ \frac{v^2}{1-2c(1-v^2)}}}{(1-2c(1-v^2))^{\frac n2 +1}}  dv.
\end{align*}
Hence, \ref{prop:mble3} holds.

Finally we show \ref{prop:mble5}.
By (\ref{nv}) we have
\begin{align*}
&\langle DF, D(-L)^{-1} F\rangle _H\\
&= n^2 \int_{0} ^{\infty} e^{-u}du {\bf E}'\left[ W(h) ^{n-1} \left( e^{-u} W(h) + \sqrt{ 1- e ^{-2u} } W'(h)\right) ^{n-1} \right]\\
&= n^2 W(h)^{n-1} \int_0^1 dy {\bf E}'\left[ \left( yW(h)+ \sqrt{1-y^{2} } W'(h)\right) ^{n-1} \right] .
\end{align*}
where $W'(h)$ denotes an independent copy of $W(h)$, and $W(h)$ and $W'(h)$ are defined on a product of abstract Wiener spaces 
 $\left( \Omega \times \Omega, H \otimes H, P\times P'\right)$ and ${\mathbf E}'$ denotes the expectation with respect to the Gaussian measure $P'$.
Hence,
\begin{align*}
&\langle DF, D(-L)^{-1} F\rangle _H \\
&= n^{2} W(h) ^{n-1} \int_0^1 dy \sum_{j=0} ^{n-1} \frac{(n-1)!}{j! \, (n-1-j)!} (yW(h)) ^{n-1-j} {\bf E}' \left[ \left( \sqrt{1-y^{2}}W'(h)\right) ^{j} \right]\\
&= n^{2} \sum_{l=0} ^{\lfloor (n-1)/2 \rfloor} \frac{(n-1)! (2l-1)!!}{(2l)! \, (n-1-2l)!} W(h)^{2n-2-2l} \int_0^1 y^{n-1-2l} (1-y^{2} ) ^l dy
\end{align*}
where we regard $(-1)!! =1$.
Noting that the beta function $\beta (a,b)$ for $a,b>0$ satisfies
\begin{align*}
\beta (a,b) &= \int_0^1 t^{a-1} (1-t)^{b-1} dt \\
&= 2 \int_0^1 s^{2a-1} (1-s^2)^{b-1} ds,
\end{align*} 
we obtain
\begin{align*}
&\langle DF, D(-L)^{-1} F\rangle _H \\
&= \frac{n^{2}}{2} \sum_{l=0} ^{\lfloor (n-1)/2 \rfloor} \frac{(n-1)! (2l-1)!!}{(2l)! \, (n-1-2l)!} \beta \left( \frac n2 -l,l+1\right) \left| F+{\bf E}\left[ W(h)^n\right] \right| ^{2-2(l+1)/n}.
\end{align*}
\end{proof}

Now we give several common probability distributions as examples.

\begin{example}\label{exam1} (Gaussian and Gamma) \rm
Let $F = cW(h)$ where $c \in {\mathbb R}$ and $h\in H$ such that $\| h\| _H =1$.
Then, ${\mathbf E}[F]=0$.
By \ref{prop:mble1} of Proposition \ref{prop:mble}, we have  $\langle D(-L)^{-1}F, DF\rangle _H$ is measurable with respect to the $\sigma$-field generated by $F$, and we have
\[
\frac{1}{2} a(x) = c^{2}, \quad x \in {\mathbb R}.
\]

Let $F = c(W(h)^2-1)$ where $c \in {\mathbb R}$ and $h\in H$ such that $\| h\| _H =1$.
Then, ${\mathbf E}[F]=0$.
By \ref{prop:mble1.5} of Proposition \ref{prop:mble}, we have $\langle D(-L)^{-1}F, DF\rangle _H$ is measurable with respect to the $\sigma$-field generated by $F$, and
\[
\frac{1}{2} a(x)= 2 c^{2} (x+1), \quad x>-1.
\]
\end{example}

\begin{example}(Log-normal) \rm 
Let $F=e^{cW(h)}$ where $c \in {\mathbb R} \setminus \{ 0\}$ and $h\in H$ such that $\| h\| _H =1$.
By \ref{prop:mble2} of Proposition \ref{prop:mble}, we have $\langle D(-L)^{-1}(F-{\mathbf E}[F]), DF\rangle _H$ is measurable with respect to the $\sigma$-field generated by $F$, and
\[
\frac{1}{2} a(x)= c^2 x\int _0^1 x^v e^{\frac {c^2}2(1-v^2)} dv, \quad x>0.
\]
\end{example}

\begin{example}(Exponential of the sum of Gaussian squares) \rm 
Let $F= e^{c\sum _{k=1}^n W(h_k) ^2}$ where $n\in {\mathbb N}$, $c \in (-\infty , 1/2)  \setminus \{ 0\}$ and $h_1,h_2,\dots ,h_n \in H$ such that $\| h_k\| _H =1$ for $k=1,2,\dots ,n$, and $\langle h_k, h_l\rangle _H =0$ for $k,l =1,2,\dots ,n$.
By \ref{prop:mble3} of Proposition \ref{prop:mble}, we have  $\langle D(-L)^{-1}(F-{\mathbf E}[F]), DF\rangle _H$ is measurable with respect to the $\sigma$-field generated by $F$ and
\[
\frac{1}{2}a(x)=4c x(\log x) \int _0^1 \frac{vx^{ \frac{v^2}{1-2c(1-v^2)}}}{(1-2c(1-v^2))^{\frac n2 +1}}  dv
\]
for $x$ belonging to the support of the law (that depends on the choice of $c$).
This case includes the uniform distribution $U[0,1]$ (by taking for example $n=2$ and $c=-\frac{1}{2}$), the centered  Pareto distribution (if we consider $n=2$ and $c=\frac{1}{4}$ we obtain a centered Pareto distribution with parameter 2) or the centered beta distribution (by taking $c=-1$, $n=2$ we have a random variable with centered  beta law with  parameters $\frac{1}{2}$ and $1$). We refer to \cite{KuTu} for a review on these probability distributions. 
\end{example}

\begin{example}(Power of  Gaussian) \rm
Consider the random variable
\[
G= W(h)^n -{\bf E}\left[ W(h)^n \right]
\]
where $h\in H$ such that $\| h\| _H=1$ and $n\in {\mathbb N}$.
By \ref{prop:mble5} of Proposition \ref{prop:mble}, we have  $\langle D(-L)^{-1}(F-{\mathbf E}[F]), DF\rangle _H$ is measurable with respect to the $\sigma$-field generated by $F$ and
\begin{align*}
&\frac{1}{2}a(x) \\
&= \frac{n^2}{2} \sum_{l=0} ^{\lfloor (n-1)/2 \rfloor} \frac{(n-1)! (2l-1)!!}{(2l)! \, (n-1-2l)!} \beta \left( \frac n2 -l,l+1\right) \left| x+{\bf E}\left[ W(h)^n\right] \right| ^{2-2(l+1)/n}
\end{align*}
for $x$ in the support of the law of $G$.
The case that $n=1$ and $n=2$ are associated to the standard normal distribution and the centered Gamma distribution, respectively (see Example \ref{exam1}).
When $n=3$,
\[
\frac{1}{2}a(x) = 3|x|^{4/3} + 6 |x|^{2/3}, \quad x \in {\mathbb R}.
\]
When $n=4$,
\[
\frac{1}{2}a(x) = 4(x+3)^{3/2} + 12 (x+3), \quad x\in (-3, \infty ).
\]
\end{example}

\begin{remark}
A concrete example of application of the theory presented in this section is given in \cite{T3} and concerns the  convergence of a sequence of Pareto distributed random variables to the exponential law. Indeed, if we consider a sequence of random variables $(F_{n}) _{n \geq 1} $ such that each  $F_{n}$ follows a Pareto distribution with parameter $n \geq 2$ then it  is well-known that the sequence $(Y_{n}) _{n\geq 1}$ given by $ Y_{n}= (n-1)( F_{n}-1)$ converges in law, as $n \to \infty$, to the exponential distribution with parameter 1. The rate of convergence has been estimated  in \cite{T3} using the techniques presented in this section.
\end{remark}

\section{Extension of the Fourth Moment Theorem}

In this section, we analyze the weak   convergence of sequences of random variables in a Wiener chaos  to $X\sim \mu$.
The purpose of the argument is to generalize the Fourth Moment Theorem (Theorem \ref{th:4momO}).

Through this section, $\mu$ is assumed to satisfy the hypothesis in the previous section. In addition we set $b(x):=-x$ and assume that  $\int a(x)^2 \mu(dx) <\infty$ and the expectation of the invariant measure equals zero, i.e. $\int x \mu (dx) =\mathbf{E}[X]=0$.

Let us assume that   $F_{m}= I_{p} (f_{m}), m\in {\mathbb N}$ is sequence of multiple integrals in the $p$th Wiener, with $f_{m} \in H ^{\odot p} $ for every $m \in {\mathbb N}$. Our result in Theorem \ref{thm:4momG} says that the joint convergence in distribution as $m\to \infty$ of $(F_{m}, \frac{1}{p} \| DF_{m} \| _H^{2} )$  to $(G, \langle DG, D(-L) ^{-1} G\rangle _H)$, with $G\sim \mu$, is equivalent to the convergence in $L^{2}(\Omega)$  of $\frac{1}{2} a(F_{m}) -\frac{1}{p} \Vert DF_{m} \Vert _H^{2} $  to zero as $m\to \infty$, where $a$ denotes the squared diffusion coefficient associated to the law $\mu$.

The nice results in \cite{NuPe} and \cite{NoPe4} show that, in the case when the target distribution $\mu$ is Gaussian or (centered) Gamma, then the hypothesis on the convergence of $\Vert DF_{m} \Vert _H$ (or the conditional expectation in Theorem \ref{thm:conv}) can be eliminated. Indeed, by assuming that $\mathbf{E} F_{m} ^{2} \to _{m} \mathbf{E} G ^{2}$, the sequence $(F_{m})_{m\geq 1}$ convergence to $\mu$  (which is Gaussian or Gamma) if and only if   
the sequence  $\frac{1}{2} a(F_{m}) -\frac{1}{p} \Vert DF_{m} \Vert _H^{2} $  converges  to zero in $L^{2} (\Omega)$. 

It is then natural to ask if the same holds for other target distributions. We will show in the sequel  that the answer is in general negative. We will show that in some special situations, the convergence in law of a sequence of multiple integrals to a law (different from Gaussian and Gamma) can be expressed in terms of the convergence of sequence of the fourth moment and also in terms of some squared diffusion coefficients associated to the law and of the Malliavin derivatives.

 \subsection{Product normal distribution}
Let $X,Y\sim N(0,1) $ be two independent random variables ($X=I_{1}(h_{1}), Y= I_{1} (h_{2})$ with $h_{1}, h_{2} \in H$ orthonormal) and consider the random variable 
\[
G=XY.
\]
Then $G$ follows the so-called product normal distribution and its density is given by 
\[
f_G(x)= K_{0}(\vert x\vert ), \quad x\in {\mathbb R}
\]
where $K_{0}$ is the modified Bessel function of the second kind.
We remark that $f_G(x)$ diverges at $x=0$, in particular $f_G$ is not continuous in ${\mathbb R}$.
Hence, Theorems in Sections \ref{sec:review} and \ref{sec:NScond} are not applicable in this case.

We will characterize the convergence in distribution (which is equivalent to the convergence in total variation) of a particular class of sequences of multiple integrals toward the product normal distribution.

\begin{remark}\label{23i-8}
Let us recall that, if $X,Y$ are independent multiple integrals, then $(X, DX)$ and $(Y, DY)$ are independent random vectors.  Moreover, $\langle DX, DY\rangle _H =0$.  See Lemma 1 in \cite{BoTu}.
\end{remark}

\begin{theorem}\label{23i-5}
Suppose $(X_{k}= I_{p}(f_{k})) _{ k\in {\mathbb N}}$ and $ (Y_{k} = I_{q} (g_{k}))_{k\in {\mathbb N}}$ are two  sequences of multiple integrals in the $p$th and $q$th Wiener chaos respectively such that
\[
\lim _{k\rightarrow \infty} {\bf E}[ X_k^2] = 1 \quad \mbox{and}\quad \lim_{k\rightarrow \infty}{\bf E}[Y_k^2] = 1.
\]
Assume that for every $k \in {\mathbb N}$, the random variables $X_{k}$ and $Y_{k}$ are independent. Then the following are equivalent:
\begin{enumerate}
\item \label{23i-5-1}  $X_{k}Y_{k} $ converges in distribution as $k\to \infty$ to $XY$.

\item \label{23i-5-2} $\lim _{k\rightarrow \infty} {\bf E}\left[ ( X_{k} Y_{k}) ^4\right] = 9.$

\item \label{23i-5-3} $\lim_{k\rightarrow \infty} \Vert f_{k} \otimes _{r} f_{k} \Vert _{H^{\otimes 2(p-r)}}=0 $ and $\lim_{k\rightarrow \infty} \Vert g_{k} \otimes _{l} g_{k}\Vert _{H^{\otimes 2(q-l)}}= 0$ for every $r=1,2,\dots , p-1$ and $ l=1,2,\dots , q-1$.

\item \label{23i-5-4} $\lim_{k\rightarrow \infty} \Vert D X_{k} \Vert _H^{2}= p$ and $\lim_{k\rightarrow \infty} \Vert DY_{k} \Vert _H ^{2} = q$ in $L^{2}(\Omega)$.

\item \label{23i-5-5} $\lim_{k\rightarrow \infty} {\bf E}[ X_k^4] = 3$  and $\lim_{k\rightarrow \infty} {\bf E}[Y _k^4] = 3$.

\item \label{23i-5-6} $X_{k} \to _{k\to \infty} X$ and $Y_{k} \to _{k\to \infty} Y$ in distribution.

\end{enumerate}
\end{theorem}

\begin{proof}
Note that  for every $k\in {\mathbb N}$,  $X_{k} Y_{k}= I_{p+q} ( f_{k} \otimes g_{k}) $. This  comes from the product formula for multiple stochastic integrals (\ref{eq:Nualart2}) and the fact that, for every $k\geq 1$, the independence of $X_{k}$ and $ Y_{k}$ is equivalent to $f_{k} \otimes _{1} g_{k} =0$  almost everywhere, see \cite{UZ}.
Note that the implication \ref{23i-5-6} $\to$ \ref{23i-5-1} is trivial, and the equivalence between \ref{23i-5-3}, \ref{23i-5-4}, \ref{23i-5-5} and \ref{23i-5-6} follows from Theorem \ref{th:4momO}.
Hence, it is sufficient to prove
\[
\ref{23i-5-1} \to \ref{23i-5-2} \to \ref{23i-5-3}
\]
for the first assertion.
The implication \ref{23i-5-1} $\to$ \ref{23i-5-2} is obvious since $X_{k}Y_{k} $ is a multiple integral in the Wiener chaos of order $p+q$ and thus for every $r\geq 1$,  $\sup_{k} \mathbf{E} \left[ \vert X_{k} Y_{k} \vert ^{r} \right] <\infty$. Let us show \ref{23i-5-2} $\to$ \ref{23i-5-3}. We use formula (\ref{m4}) for the fourth moment of a multiple integral and we get
\[
{\bf E}\left[ I_{n} (f) ^{4}\right] = 3 {\bf E}\left[ I_{n}(f) ^{2} \right] ^2 + R_{f}
\]
where, for every $f\in H^{\odot n}$, we denoted
\begin{equation}
\label{rf}
R_{f}:= n! ^{2} \sum_{r=1} ^{n-1} \left(\frac{n!}{r!\, (n-r)!}\right) ^{2} \left[ \Vert f\otimes _{r} f \Vert _{H^{\otimes 2(n-r)}}^{2} + C_{2n-2r}^{n-r} \Vert f\tilde{\otimes} _{r} f\Vert  _{H^{\otimes 2(n-r)}}^{2} \right].
\end{equation}
We have, for every $k\geq 1$,
\begin{eqnarray*}
{\bf E}\left[ (X_{k} Y_{k} ) ^4 \right] &=& {\bf E}\left[ X_k^4\right] {\bf E}\left[ Y_k^4\right] \\
&=& (3 {\bf E}\left[X_k^2 \right]^2 + R_{f_{k}} ) (3{\bf E}\left[ Y_k^2\right]^2 + R_{g_{k}} )
\end{eqnarray*}
with $R_{f_{k}}$ and $R_{g_{k}}$ defined above by (\ref{rf}). Since $3{\bf E}\left[ X_k^2\right] ^2 \times 3{\bf E}\left[ Y_k^2\right] ^{2} \to 9$,  we easily get that $\Vert f_{k} \otimes _{r} f_{k} \Vert  _{H^{\otimes 2(p-r)}} \to _{k\to \infty} 0 $ and $\Vert g_{k} \otimes _{l} g_{k}\Vert  _{H^{\otimes 2(q-l)}} \to_{k \to \infty} 0$ for every $r=1,..., p-1$ and $ l=1,.., q-1$.
\end{proof}

\begin{remark}
The convergence of Wiener polynomials whose orders are uniformly dominated from above to a non-degenerate distribution in distribution is equivalent to the convergence in total variation (see Theorem 3.1 in \cite{NoPo}).
\end{remark}

\subsection{The sum of a standard normal random variable and an independent random variable}

Let $Z\sim N(0,1)$ and let $G$ be a centered random variable with the fourth moment and is independent of $Z$. We will characterize the convergence (of a special class of sequences of multiple stochastic integrals) to the distribution of the random variable $Z+G$. See \cite{BoTu} for the case when $Z,G$ are independent Gamma distributed random variables.
We assume that $G$ has a continuous density function on $\mathbb R$ (not necessary strictly positive on ${\mathbb R}$).
Below $a$ will denote the squared diffusion coefficient associated to $Z+G$.

\begin{prop}\label{23i-10}
 Let $(X_{k}= I_{p}(f_{k}))_{k\in {\mathbb N}}$ be a sequence of random variables in the $p$th Wiener chaos such that
\[
\lim _{k\to \infty} \mathbf{ E}[X_k^2] =1.
\]
Let $(Y_{k}=(I_{q}(g_{k}))_{ k\in {\mathbb N}}$ be another sequence of random variables in the $q$th Wiener chaos such that $X_k$ and $Y_k$ are independent for every $k\in {\mathbb N}$ and
\begin{equation*}
\lim _{k\rightarrow \infty} {\bf E}[Y_k^2] = {\bf E}[G^2]\quad \mbox{and}\quad \lim _{k\rightarrow \infty}{\bf E}[Y_k^4] = {\bf E}[ G^4].
\end{equation*}
Then the following assertions are equivalent:
\begin{enumerate}
\item \label{23i-10-1} $X_{k} \to _{k} Z$ and $ Y_{k}\to _{k} G$ in distribution 
\item \label{23i-10-2} $X_{k}+ Y_{k} \to  _{k}Z+ G$ in distribution
\end{enumerate}
If in addition, $a$ satisfies the assumptions in Corollary \ref{cor:conv}, then the following is also equivalent.
\begin{enumerate}
\item[\rm (iii)] \label{23i-10-3} $\displaystyle \frac{1}{2} a(X_{k} +Y_{k}) - {\bf E} \left[ \left. \frac{1}{p}\Vert DY_{k} \Vert _H^{2}+  \frac{1}{q}\Vert DX_{k} \Vert _H^{2} \right| X_{k}+Y_{k}\right] \to _{k} 0 \mbox{ in } L^{1}(\Omega).$
\end{enumerate}
\end{prop}
\begin{proof}
Clearly \ref{23i-10-1} $\to$ \ref{23i-10-2}.
We show that \ref{23i-10-2} $\to$ \ref{23i-10-1}.
We use the relation
\begin{eqnarray}
{\bf E}\left[ (X_k + Y_k)^4\right]&=& {\bf E}\left[ X_k^4\right]+ 6{\bf E}\left[X_k^2\right] {\bf E}\left[ Y_k^2\right] + {\bf E}\left[ Y_k^4 \right] \nonumber\\
&=& 3 {\bf E}\left[ X_k^2\right] +6 {\bf E}\left[ X_k^2\right] {\bf E}\left[ Y_k^2\right] + {\bf E}\left[ Y_k^4\right] + R_{f_{k}}\nonumber\\
&\sim &  3 + 6 {\bf E}\left[ G^2\right] + {\bf E}\left[ G^4\right] + R_{f_{k}}, \label{14f-6}
\end{eqnarray}
where we used (\ref{m4}) and $R_{f_{k}}$ given by (\ref{rf}). Recall that $\sim$ means that the sides have the same limit as $k\to \infty$. On the other hand, from \ref{23i-10-2} we have 
\begin{equation}
\label{14f-7}
{\bf E}\left[ (X_{k}+ Y_{k}) ^4\right] \to _{k} {\bf E}\left[ (Z+G)^4\right] = 3+ 6{\bf E}\left[G^2\right] + {\bf E}\left[G^4\right]. 
\end{equation}
By combining (\ref{14f-6}) and (\ref{14f-7}), we get  $R_{f_{k}}\to _{k} 0$ and thus
$$\Vert f_{k} \otimes _{r} f_{k} \Vert _{H^{\otimes 2(p-r)}}\to _{k} 0$$
for every $r=1,2, \dots , p-1$. This gives the converges in distribution of $X_{k}$ to $Z$ as $k\to \infty$ (see Theorem \ref{th:4momO}), and then we  easily get the convergence in law as $k\to \infty$ of $Y_{k}$ to $G$.

Theorem 3.1 in \cite{NoPo} implies that \ref{23i-10-2} is equivalent to the convergence of $X_k + Y_k$ in total variation.
Hence, if we additionally give the assumptions in Corollary \ref{cor:conv}, we obtain the equivalence between \ref{23i-10-2} and (iii) by Corollary \ref{cor:conv}.
\end{proof}

We apply the above result to the particular case when $G$ follows a centered chi-squared distribution.

\begin{corollary}
Consider  two sequences  $(X_k)_{k\geq 1}$ and $(Y_k) _{k\geq 1}$ as in Proposition \ref{23i-10} and assume that the random variable $G$ from Proposition \ref{23i-10} is a  centered Gamma random variable i.e $G= W(h)^2-1$ where $h\in H$, $\Vert h\Vert _H =1$.
Then the following are equivalent:
\begin{enumerate}
\item $X_{k}+  Y_{k} \to _{k} Z+ G$ in distribution.

\item $X_{k}\to _{k} Z$ and $ Y_{k} \to _{k} G$ in distribution.

\item $\frac{1}{p} \Vert  DX_{k} \Vert _H^{2} \to 1$ and $\frac{1}{q} \Vert DY_{k} \Vert _H^{2} - 2Y_{k} \to _{k} 0$  in $L^{2} (\Omega)$.

\item$1 -\frac{1}{p}\Vert DX_{k}\Vert _H^{2} + 2Y_{k} -\frac{1}{q} \Vert DY_{k} \Vert _H^{2} \to _{k} 0$ in $L^{2}(\Omega)$.

\item $\frac{1}{2} a(X_{k}+ Y_{k}) - {\bf E}\left[ \left.\frac{1}{p}\Vert  DX_{k} \Vert _H^{2} + \frac{1}{q}\Vert DY_{k} \Vert _H^{2} \right| X_{k}+ Y_{k}\right] \to _{k}0$ in $L^{1} (\Omega)$.
\end{enumerate}
\end{corollary}
\begin{proof}
The points (i) and (v) are equivalent due to Proposition \ref{23i-10} by noting the squared diffusion coefficient $a$ associated with the law $Z+ G$ satisfies the assumptions from Corollary \ref{cor:conv}. The fact that (i) is equivalent to (ii), (iii), (iv) follows from Theorem 1.2 in \cite{NoPe4} and Remark \ref{23i-8}. On the other hand,  the explicit expression for $a$ (that can be calculated via (\ref{n1}), is pretty complexe.
\end{proof}

\begin{remark}
Condition $(iv)$ above show that the convergence toward the sum $Z+G$ of a Gaussian and an independent Gamma distribution can be characterized  (without the appearance of the conditional expectation) in terms of the diffusion coefficients  associated with $Z$ and $G$.  
\end{remark}

\section{The case when  the diffusion coefficient is a polynomial of second degree}

Since we are studying the convergence of a sequence of multiple stochastic integrals, whose expectation is zero, we will assume that the measure $\mu$ is centered and the drift coefficient is $b(x)=-x$. We will also assume that the diffusion coefficient is a polynomial of second degree expressed as
\begin{equation}
\label{apol}
a(x)= \alpha x^{2}+ \beta x + \gamma, \hskip0.5cm x\in S,\ \alpha, \beta, \gamma  \in \mathbb{R}
\end{equation}
such that $a(x)>0$ for every $x\in S$.  In this case it is possible to understand better when the necessary and sufficient condition for the weak convergence of a sequence of multiple integrals  toward the law $\mu$ is satisfied.

This class contains the known continuous probability distributions. See Table 1 in \cite{BSS}. It contains among others, the normal, Gamma, uniform, Student, Pareto, inverse Gamma or F distributions. We mention by the way that, except the first laws listed here, the others cannot be limits in distribution of a sequence of multiple integrals, because they do not admit moments of any order. Nevertheless, the class satisfying (\ref{apol}) and admitting moments of any order contains an infinity of probability distributions.

We give below some examples qualified as candidate to be limit in distribution of sequences of multiple integrals. Note that in Table 1 in \cite{BSS} the diffusion coefficient is given for non centered probability distributions. In order to obtain the diffusion coefficient $a$ of the centered measure from the diffusion coefficient $a_{0}$ of the non centered measure we apply the following rule (see Lemma 2.5 in \cite{BSS})

\begin{equation*}
a(x)= a_{0} (x+ {\bf E}[X])
\end{equation*}
where $X \sim \mu$.

\begin{example}\label{normal}
{\bf The normal distribution $N(0,\gamma)$, $\gamma >0$. } In this case $a(x)= 2 \gamma$.
\end{example}

\begin{example}
\label{gamma}
{\bf The Gamma $\Gamma (a, \lambda)$, $a, \lambda >0$ law. } Here  the density is $f(x)= \frac{\lambda ^{a}}{\Gamma (a)} x^{a-1} e^{-\lambda x} $ for $x>0$ and $f(x) =0$ otherwise. Also  ${\bf E}[X] = \frac{a}{\lambda}$ and the centered Gamma law has
\begin{equation*}
a(x)= \frac{2}{\lambda} (x+ \frac{a}{\lambda})
\end{equation*}
meaning that $\alpha =0, \beta = \frac{2}{\lambda}, \gamma = \frac{2a}{\lambda ^{2}}$.
\end{example}

\begin{example}
\label{uniform}
{\bf The uniform $U(0,1)$ distribution. } Here the density is $f(x)= 1_{[0,1]}(x)$, the mean is ${\bf E}[X]= \frac{1}{2}$ and $U[0,1]-{\bf E}[U[0,1]]$ has squared diffusion coefficient 
$$a(x)= (x+ \frac{1}{2}) (\frac{1}{2} -x)= \frac{1}{4} -x^{2}.$$
So $\alpha= -1, \beta = 0, \gamma = \frac{1}{4}.$
\end{example}

\begin{example}
\label{beta}
{\bf The Beta $\beta (a, b)$ law, $a,b>0$. }
In this case the probability density function is 
\begin{equation*}
f(x)= \frac{\Gamma (a+b)} { \Gamma (a) \Gamma (b)} x^{a-1} (1-x) ^{b-1} 1_{(0,1)}(x), 
\end{equation*} ${\bf E}[X]= \frac{a}{a+b}$ and the centered  beta law has 
$$a(x)= \frac{2}{a+b} (x+ \frac{a}{a+b} ) (\frac{b}{a+b} -x).$$
Note that the beta distribution with $a=b$ has the fourth cumulant negative and therefore it cannot be limit of a sequence of multiple integrals.

Therefore $\alpha= -\frac{2}{a+b}, \beta = \frac{2}{a+b} \frac{b-a}{a+b}, \gamma = \frac{2}{a+b}\frac{a}{a+b}\frac{b}{a+b}.$

\end{example}

The polynomial form of the diffusion coefficient $a$ together with Theorem \ref{char2} implies several consequences on the moments of the probability distribution of $X\sim \mu$.  We will present them in  the following useful lemma. In particular, we give a recurrence formula  for the moments of  $X$.  

\begin{lemma}
Suppose that $a$ is given by (\ref{apol}).  Then for every $k\in \mathbb{R}$, $k\geq 1$ such that ${\bf E}[X^{2k}]<\infty$ one has
\begin{equation}
\label{monrec}
\left( 1- \frac{2k-1}{2} \alpha \right) {\bf E}[ X^{2k}] = \frac{2k-1}{2} \beta {\bf E}[X^{2k-1}] +\frac{2k-1}{2}\gamma {\bf E}[X^{2k-2}].
\end{equation}
In particular, if $\alpha \not=2$,

\begin{equation}\label{e2}
{\bf E}[X^{2}]= \frac{\gamma} {2-\alpha},
\end{equation}
if $\alpha \not= 1, 2$, 
\begin{equation}\label{e3}
{\bf E}[X^{3}] = \frac{\beta}{1-\alpha} {\bf E}[X^{2}]= \frac{ \beta \gamma }{(1-\alpha) (2-\alpha ) }
\end{equation}
and if $\alpha \not= 2, \frac{2}{3}$, then 
\begin{equation}\label{e4}
{\bf E}[X^{4}] = \frac{3 (\frac{\beta ^{2}}{1-\alpha } + \gamma )}{2-3\alpha} {\bf E}[X^{2}] = \frac{3 \gamma (\frac{\beta ^{2}}{1-\alpha } + \gamma )}{(2-\alpha) (2-3\alpha)}.
\end{equation}
\end{lemma}
\begin{proof} Relation (\ref{monrec}) is obtained by applying Theorem \ref{char2} with $h(x)= x^{2k-1}$.  In particular, for $k=1$, since ${\bf E}[X]=0$
$${\bf E}[X^{2}]= {\bf E}\left[ \frac{1}{2} a(X)\right] = \frac{1}{2} \left( \alpha {\bf E}[X^{2}] + \beta {\bf E}[X] + \gamma \right) $$
and thus ${\bf E}[X^{2}]= \frac{\gamma} {2-\alpha}$.
By applying successively  Proposition \ref{char2} with $h(x)= x^{2}$ ($k=\frac{3}{2}$) and $h(x)= x^{3}$ ($k=2$) we obtain the expressions (\ref{e3}) and (\ref{e4}).
\end{proof}

\begin{lemma}\label{l2} Assume (\ref{apol}) with  $\alpha \not=1,2, \frac{2}{3}$. Fix $n\geq 1$ and let $\{ F_{m}= I_{n} (f_{m}), m\geq 1\} $ such that  
\begin{equation}
\label{m2}
{\bf E}[ F_{m} ^{2}] \to _{m\to \infty} {\bf E}[X^{2}],\quad  {\bf E}[F_{m} ^{4}] \to _{m\to \infty}{\bf E}[X^{4}] \mbox { and } {\bf E}[F_{m} ^{3}] \to _{m\to \infty} {\bf E}[X^{3}].
\end{equation}

Then 
 \begin{eqnarray}
&&{\bf E}\left[ F_{m} ^{4} -\frac{3}{2} a(F_{m}) F_{m} ^{2} \right]\nonumber\\
&\simeq & C_{0} n! \Vert f_{m} \Vert ^{2} + 3n \sum_{p=1; p\not= \frac{n}{2}} ^{n-1} (p-1) ! \left( \begin{array}{c} n-1\\ p-1\end{array}\right)^{2} p! \left( \begin{array}{c} n\\ p\end{array}\right) ^{2} (2n-2p)! \Vert f_{m} \tilde{ \otimes }_{p} f_{m} \Vert ^{2}\nonumber \\
&&+ \frac{3}{2} c_{n} ^{-2}n!  \Vert f_{m} \tilde{ \otimes }_{n/2} f_{m} \Vert ^{2} -\frac{3}{2} \beta {\bf E}[ F_{m} ^{3}] \to _{m\to \infty}0 \label{10i3}
\end{eqnarray}
where $\simeq $ means that the two sides have the same limit as $m\to \infty$, $c_{n}$ is defined by \begin{equation}
\label{cn}
c_{n}= \frac{ (n/2)! ^{3}}{n! ^{2}}
\end{equation}  
and 
\begin{eqnarray}
\label{C0}
C_{0}&=& \frac{3}{2}\alpha \left[ \frac{-4\gamma}{(2-\alpha)(2-3\alpha)} -\frac{3\beta ^{2}}{(1-\alpha) (2-3\alpha) } \right].
\end{eqnarray}

\end{lemma}
\begin{proof}
From assumption (\ref{m2}) and Theorem \ref{char2},
$$ {\bf E}\left[  F_{m} ^{4} -\frac{3}{2} a(F_{m}) F_{m} ^{2} \right] \to_{m} {\bf E}\left[ X^{4} -\frac{3}{2} a(X) X^{2}\right] =0.$$

Now,  by (\ref{apol})
\begin{eqnarray*}
 \mathbf{E} \left[ F_{m} ^{4} -\frac{3}{2} a(F_{m}) F_{m} ^{2} \right]
&=& {\bf E}[F_{m} ^{4}] -\frac{3}{2} \alpha {\bf E}[F_{m}^{4}] - \frac{3}{2} \gamma {\bf E}[F_{m}^{2}] -\frac{3}{2} \beta {\bf E}[F_{m} ^{3}]\\
 \end{eqnarray*}
and, on the other hand, using (\ref{m2}), (\ref{e4}) and (\ref{m4}) we get
\begin{eqnarray*}
&& \left( 1-\frac{3}{2}\alpha \right) {\bf E}[ F_{m} ^{4}] -\frac{3}{2} \gamma {\bf E}[F_{m} ^{2}] \\
&&\simeq  3 {\bf E}[F_{m} ^{2}] ^{2} + 3n \sum_{p=1} ^{n-1} (p-1) ! \left( \begin{array}{c} n-1\\ p-1\end{array}\right)^{2} p! \left( \begin{array}{c} n\\ p\end{array}\right) ^{2} (2n-2p)! \Vert f_{m} \tilde{ \otimes }_{p} f_{m} \Vert ^{2}\\
&&-\frac{3}{2}\alpha \frac {3(\frac { \beta ^{2}}{1-\alpha} + \gamma )}{2-3\alpha } {\bf E}[F_{m} ^{2}]  -\frac{3}{2} \gamma {\bf E}[F_{m} ^{2}]\\
&&\simeq C_{0}{\bf E}[F_{m} ^{2}] +  3n \sum_{p=1} ^{n-1} (p-1) ! \left( \begin{array}{c} n-1\\ p-1\end{array}\right)^{2} p! \left( \begin{array}{c} n\\ p\end{array}\right) ^{2} (2n-2p)! \Vert f_{m} \tilde{ \otimes }_{p} f_{m} \Vert ^{2}
\end{eqnarray*}
with $C_{0}$ given by (\ref{C0}). Therefore, the desired relation (\ref{10i3}) is obtained.
\end{proof}

 From the above result we will deduce several restrictions on the probability distributions that can be limits of sequences of multiple stochastic integrals. We will discuss separately the cases $\beta=0$ and $\beta \not=0$.

\subsection{The case $\beta =0$}

Let us assume that $\beta =0$ in (\ref{apol}).  This is the case of the Student, uniform and beta (with parameters $a=b$) distributions.  In this situation, since $\mathbf{E}X^{3} =0$ (see (\ref{e3})), the order of the chaos  $n$ can be even or odd in principle.

\begin{theorem}\label{10i2}  Assume (\ref{apol}) with  $\alpha \not=2, \frac{2}{3}$ and $\beta =0$.   Fix $n\geq 1$ and let $( F_{m}= I_{n} (f_{m})_{ m\geq 1} $ satisfying (\ref{m2}). Then $\alpha =0, \gamma >0$ and  if $S=(-\infty, \infty)$, $X$ follows a  centered normal distribution with variance $\gamma$. 

\end{theorem}
\begin{proof}
Let us show that $C_{0} \geq 0$.  Note that
$$ C_{0}= -\frac{6\alpha  \gamma}{(2-\alpha)(2-3\alpha) }.$$
 Relations (\ref{e2}) and (\ref{e4}) with $\beta =0$ give
\begin{equation*}
{\bf E}[X ^{2}]= \frac{\gamma}{2-\alpha} \mbox{ and } {\bf E}[X^{4}] = \frac{3\gamma ^{2} }{(2-\alpha)(2-3\alpha)}
\end{equation*}
and this implies in  particular
$$(2-\alpha) (2-3\alpha) >0.$$

On the other hand,  from the relation (\ref{monrec}) with $\beta =0$ we notice that
$$
\left( 1 -\frac{2k-1}{2} \alpha \right) {\bf E}[X^{2k}]= \frac{2k-1}{2} \gamma {\bf E}[X ^{2k-2}]$$
for every $k\geq 1$ we notice that $\alpha , \gamma $ have different  parities. Indeed, if $\alpha >0$ then for $k$ large enough $\left( 1 -\frac{2k-1}{2} \alpha \right)$ becomes negative and thus $\gamma <0$. If $\alpha <0$ then $\left( 1 -\frac{2k-1}{2} \alpha \right) $ is positive and so $\gamma $ should be positive. So $C_{0} \geq 0$.

Next, since  ${\bf E}\left[ F_m^{4} -\frac{3}{2} a(F_{m} ) F_{m} ^{2} \right] \to _{m} 0$ we have from  Lemma \ref{l2}
$$C_{0} {\bf E}[F_{m} ^{2}] + 3n \sum_{p=1} ^{n-1} (p-1) ! \left( \begin{array}{c} n-1\\ p-1\end{array}\right)^{2} p! \left( \begin{array}{c} n\\ p\end{array}\right) ^{2} (2n-2p)! \Vert f_{m} \tilde{ \otimes }_{p} f_{m} \Vert ^{2}\to _{m\to \infty} 0 $$
and this is not possible unless $C_{0}=0.$ In this case, we have $\alpha= 0$ (from (\ref{C0})), $\gamma >0$ (from (\ref{e2})) and consequently, when $S=(-\infty, \infty)$,     $X$ follows a normal distribution with mean zero and variance $\gamma$. This can be easily obtained by computing the density of $X$ using Proposition 1 in \cite{KuTu}.

\end{proof}

As a consequence, we notice that several probability distributions cannot be limits in distribution of sequences of multiple stochastic integrals. 
\begin{corollary}
A sequence of random variables in a fixed Wiener chaos cannot converge to the uniform distribution.
\end{corollary}

\subsection{The case $\beta \not =0$}
In this paragraph we study the convergence of a sequence a multiple stochastic integrals to the law of a random variable $X\sim \mu$ where $\mu$ is the invariant measure of a diffusion with drift coefficient $b(x)=-x$ (meaning that ${\bf E}[X]=0$) and diffusion coefficient $a$ given by (\ref{apol}) with $\beta \not =0$. This is the case of the  Pareto, Gamma, inverse Gamma and $F$ distributions.

Fix $n\geq 1$. Consider throughout this section that  $\{ F_{m}, m\geq 1\}$ is a sequence of random variables expressed as $F_{m} = I_{n} (f_{m})$ with $f_{m} \in {H} ^{\odot n}$.
Since the third moment of a multiple Wiener -It\^o integral of odd order is zero, from (\ref{e3}) we may  assume  in this paragraph that {\bf $n$ is even. }

Let us first deduce some consequences on the convergence in law of $F_{m}$ to $X$.

\begin{lemma}\label{l11}
 Assume (\ref{apol}) with  $\alpha \not=1, 2, \frac{2}{3}$.   Fix $n\geq 1$ and let $\{ F_{m}= I_{n} (f_{m}), m\geq 1\} $ satisfying (\ref{m2}). Then
\begin{equation*}
\lim _{m} \langle f_{m}, f_{m} \tilde{ \otimes } _{n/2}f_{m} \rangle =\lim _{m} \frac{\beta}{1-\alpha } c_{n}  \Vert f_{m} \Vert ^{2}
\end{equation*}
where $c_{n}$ is given by (\ref{cn}). 
\end{lemma}
\begin{proof}
Condition (\ref{m2}) and  Proposition \ref{char2} imply 
\begin{equation*}
\mathbf{E} \left[ F_{m} ^{3} -a(F_{m} ) F_{m} \right] \to _{m\to \infty }{\bf E}[ X^{3} -a(X)X] =0
\end{equation*}
or equivalently
\begin{equation*}
(1-\alpha) {\bf E}[F_{m} ^{3}] -\beta {\bf E}[F_{m} ^{2}] \to _{m} 0.
\end{equation*}
But from (\ref{m3}) we have
\begin{eqnarray*}
(1-\alpha) {\bf E}[F_{m} ^{3}] -\beta {\bf E}[F_{m} ^{2}]
&=& (1-\alpha)  \frac{n! ^{3}}{(n/2) ! ^{3}}\langle f_{m}, f_{m }\tilde{\otimes} _{\frac{n}{2}} f_{m} \rangle -\beta n!  \Vert f_{m} \Vert ^{2}\\
&= &    \frac{n! ^{3}}{(n/2) ! ^{3}} \left[ (1-\alpha) \langle f_{m}, f_{m }\tilde{\otimes }_{\frac{n}{2}}f_{m} \rangle - \beta c_{n}  \Vert f_{m} \Vert ^{2}\right]
\end{eqnarray*}
and since this converges to zero as $m\to \infty$, we obtain the conclusion. Let us also mention that the two limits in the statement of the lemma exist due to   Lemma \ref{8i1} and to the convergence of the sequences of the second and third moments of $(F_{m}) _{ m\geq 1}$. 
\end{proof}

We will need the following auxiliary lemma.

\begin{lemma}\label{l22} Let the assumptions stated in Lemma \ref{l11} prevail.  
Then for every $c \in \mathbb{R}$ 
\begin{eqnarray}
&&\mathbf{E}\left[ F_{m} ^{4} -\frac{3}{2} a(F_{m}) F_{m} ^{2} \right]\nonumber\\
&\simeq &  \frac{3}{2} c_{n} ^{-2} n! \left[ \frac{2}{3}c_{n}^{2}\left( C_{0} -\frac{3}{2} (1-c)  \frac{\beta ^{2}}{1-\alpha } \right) \Vert f_{m} \Vert ^{2} -\beta c c_{n}\langle f_{m}, f_{m} \tilde{\otimes}_{\frac{n}{2}} f_{m} \rangle+ \Vert f_{m} \tilde{ \otimes }_{n/2} f_{m} \Vert ^{2}\right] \nonumber\\
&&+3n \sum_{p=1; p\not= \frac{n}{2}} ^{n-1} (p-1) ! \left( \begin{array}{c} n-1\\ p-1\end{array}\right)^{2} p! \left( \begin{array}{c} n\\ p\end{array}\right) ^{2} (2n-2p)! \Vert f_{m} \tilde{ \otimes }_{p} f_{m} \Vert ^{2}\nonumber \\
&&\to_{m\to \infty } 0 \label{10i5}
\end{eqnarray}
where $\simeq $ means that the two sides have the same limit as $m\to \infty$, $c_{n}, C_{0}$ are defined by (\ref{cn})  and  (\ref{C0}) respectively.
\end{lemma}
\begin{proof} Recall the relation (\ref{10i3}). For every $c\in \mathbb{R}$ we can write 

$$\frac{3}{2} \beta {\bf E}[F_{m}^3] =\frac{3}{2}(1-c) \beta {\bf E}[F_{m} ^{3}] + \frac{3}{2}c \beta {\bf E}[F_{m} ^{3}]$$
and using (\ref{e3}) and the convergence of the moments of $F_{m}$ to those of  $\mu$ (relation (\ref{m2})), 
\begin{eqnarray}
&&\frac{3}{2}(1-c) \beta {\bf E}[F_{m} ^{3}] \simeq \frac{3}{2}(1-c) \beta \frac{\beta}{1-\alpha} {\bf E}[X^{2}] \nonumber\\
&\simeq &\frac{3}{2}(1-c) \beta \frac{\beta}{1-\alpha} {\bf E}[F_{m}^{2}] =\frac{3}{2}(1-c) \beta \frac{\beta}{1-\alpha}  n! \Vert f_{m} \Vert ^{2}.\label{10i4}
\end{eqnarray}
By combining (\ref{10i3}) and (\ref{10i4}), we obtain (\ref{10i5}).

\end{proof}

The next step is to find $c\in \mathbb{R}$ such that 
$$\frac{2}{3}\left( C_{0} -\frac{3}{2} (1-c)  \frac{\beta ^{2}}{1-\alpha } \right) = A^{2} \mbox { and } \beta c=2A.$$
If such $c\in \mathbb{R}$ exists then the sequence ${\bf E}[ F_{m} ^{4} -\frac{3}{2} a(F_{m}) F_{m} ^{2} ]$ (which converges to zero as $m\to \infty$ due to (\ref{eq:char2})) will have the same limit as $\Vert Ac_n f_{m}- f_{m} \tilde{\otimes}_{n/2} f_{m}\Vert ^{2}$ plus a positive term (this is consequence of (\ref{10i5})). 
The existence of $c\in \mathbb{R}$ is equivalent to the existence of a real solution to the second degree  equation  
\begin{equation}
\label{ec}
3\beta ^{2} c ^{2} - \frac{12 \beta ^{2} }{1-\alpha } c -\left( 8C_{0} -\frac{12\beta ^{2}}{1-\alpha }\right)=0
\end{equation}
which has the following discriminant
\begin{equation}
\label{D}
\Delta= -\frac{144 \alpha}{2-3\alpha} \left[ \frac{\beta ^{2} }{(1-\alpha ) ^{2}} + \frac{2\gamma} {2-\alpha} \right]. 
\end{equation}
Since $ \frac{\beta ^{2} }{(1-\alpha ) ^{2}} + \frac{2\gamma} {2-\alpha} >0$ (see (\ref{e2})) the sign of $\Delta $ depends on the sign of $\frac{\alpha}{2-3\alpha}$.

\begin{theorem}\label{prop:not0}
Assume (\ref{apol}) with  $\alpha \not=1, 2, \frac{2}{3}$.   Fix $n\geq 1$ and let $\{ F_{m}= I_{n} (f_{m}), m\geq 1\} $ satisfying (\ref{m2}).
Moreover, let us assume $\frac{\alpha}{2-3\alpha} \leq 0$ (that is, $\alpha \in \mathbb{R}\setminus (0, \frac{2}{3}])$. Then $\alpha =0$ and, if $S=(-\frac{a}{\lambda}, \infty)$,  $X$ follows a centered Gamma law  $\Gamma (a, \lambda) - {\bf E}[\Gamma (a, \lambda)]$ where  $\beta=\frac{2}{\lambda}, \gamma = \frac{2a}{\lambda ^{2}}$.

\end{theorem}

\begin{proof} When $\alpha \notin [0, \frac{2}{3})$, then the discriminant $\Delta $ (\ref{D}) is positive and  the equation (\ref{ec}) admits two real solutions (that may coincide).  Let us denote by $c_{1}, c_{2} \in \mathbb{R} $ and by $A_{i}= \frac{\beta c_{i}}{2}$, $i=1,2$. From (\ref{10i5}) we deduce that, for $i=1,2$
\begin{eqnarray*}
{\bf E}\left[ F_{m} ^{4} -\frac{3}{2} a(F_{m}) F_{m} ^{2} \right] \simeq  \frac 32 c_n^{-2}n! \left| \left|  A_{i} c_{n}  f_{m} -f_{m} \tilde{ \otimes }_{n/2} f_{m}\right| \right| ^{2} \\
+ 3n \sum_{p=1; p\not= \frac{n}{2}} ^{n-1} (p-1) ! \left( \begin{array}{c} n-1\\ p-1\end{array}\right)^{2} p! \left( \begin{array}{c} n\\ p\end{array}\right) ^{2} (2n-2p)!\Vert f_{m} \tilde{ \otimes }_{p} f_{m} \Vert ^{2}\to _{m} 0
\end{eqnarray*}
and consequently 
\begin{equation}
\label{aux} \Vert A _{i} c_{n}f_{m} - f_{m} \tilde{\otimes}  _{n/2} f_{m} \Vert \to _{m} 0 \mbox{ and } \Vert f_{m} \otimes _{p} f_{m} \Vert \to _{m\to \infty} 0 \mbox{ for $p=1,..., n-1$, $p\not=\frac{n}{2}$}
\end{equation}
where $c_{n}$ is given by (\ref{cn}). Relation (\ref{aux}) and Lemma  \ref{l11} immediately imply $A_{i}= \frac{\beta}{1-\alpha }$ for $i=1,2$ and consequently the two solutions to (\ref{ec}) must coincide. The discriminant $\Delta $ (\ref{D}) then vanishes and that gives $\alpha =0$.  The fact that $\mu$ is a centered Gamma law follows from  (\ref{apol}) with $\alpha=0$ and $\beta \not=0$ by computing the density of $X$ in terms of the squared diffusion coefficient $a$, see Proposition 1 in \cite{KuTu}.

\end{proof}

Since in the case of the beta distribution $\alpha=\frac{-2}{a+b}$ and $a,b>0$, we have the following corollary.

\begin{corollary}
A sequence of random variables in a fixed Wiener chaos cannot converge to the beta distribution.
\end{corollary}

\begin{remark}
In the case of the centered Gamma distribution, we obtain from the proof of Theorem  \ref{prop:not0} :  a sequence $(F_{m}= I_{n}(f_{m}) _{m\geq 1}$ such that ${\bf E}[F_{m}^{2}] \to _{m\to \infty} \frac{a}{\lambda ^{2}}$ converges to the centered Gamma law $\Gamma (a, \lambda)- {\bf E}[\Gamma (a, \lambda)]$ if and only if the following assertions are satisfied:
\begin{description}
\item{$\bullet$ } ${\bf E}[F_{m}^{3}] \to _{m} \frac{2a}{\lambda ^{3}}$ and ${\bf E}[F_{m} ^{4}] \to _{m} \frac{3a(a+2)} {\lambda ^{4} }$
\item{$\bullet$ } $\Vert  \frac{2}{\lambda} c_{n} f_{m} -f_{m}\tilde{\otimes } _{n/2} f_{m} \Vert \to _{m} 0$ (recall that $c_{n}$ is given by (\ref{cn}))
\item{$\bullet$ } $\frac{1}{\lambda} F_{m} ^{2} + \frac{a}{\lambda ^{2}} -\frac{1}{n}\Vert DF_{m}\Vert  _{H}^{2} $ converges to zero in $L^{2}(\Omega)$.

\end{description}
When $\lambda = \frac{1}{2}$ and $a=\frac{\nu}{2}$  we retrieve the results in \cite{NoPe4}, Theorem 1.2.
\end{remark}

\end{document}